\documentclass[12pt]{amsart}
\usepackage{amsfonts,amsmath,amssymb,amsthm}
\usepackage{latexsym}
\usepackage{mathrsfs}
\usepackage{tikz}
\usepackage{color}

\newtheorem{lem}{Lemma}[section]
\newtheorem{prop}{Proposition}[section]
\newtheorem{thm}{Theorem}[section]
\newtheorem{rem}{Remark}

\newtheorem{definition}{Definition}[section]
\newtheorem{conj}{Conjecture}[section]

\newcommand{\un}{\mathbb{I}}

\newcommand{\cA}{\mathcal{A}}
\newcommand{\cB}{\mathcal{B}}

\newcommand{\diag}[3]{ \foreach \t in {1,...,#3} {\draw[thick] (#1+\t,#2-1) rectangle (#1+\t-1,#2);} }

\begin{document}

\title[Racah algebra and centralizers]{Temperley--Lieb, Brauer and Racah algebras and other centralizers of $\mathfrak{su}(2)$}


\author[N.Cramp\'e]{Nicolas Cramp\'e$^{\dagger,*}$}
\address{$^\dagger$ Institut Denis-Poisson CNRS/UMR 7013 - Universit\'e de Tours - Universit\'e d'Orl\'eans, 
Parc de Grammont, 37200 Tours, France.}
\email{crampe1977@gmail.com}

\author[L.Poulain d'Andecy]{Lo\"ic Poulain d'Andecy$^{\ddagger}$}
\address{$^\ddagger$ Laboratoire de math\'ematiques de Reims FRE 2011, Universit\'e de Reims Champagne-Ardenne, UFR Sciences exactes et naturelles,
Moulin de la Housse BP 1039, 51100 Reims, France. }
\email{loic.poulain-dandecy@univ-reims.fr}

\author[L.Vinet]{Luc Vinet$^{*}$}
\address{$^*$ Centre de recherches math\'ematiques, Universit\'e de Montr\'eal,
P.O. Box 6128, Centre-ville Station,
Montr\'eal (Qu\'ebec), H3C 3J7, Canada.}
\email{vinet@CRM.UMontreal.ca}

\begin{abstract}
In the spirit of the Schur--Weyl duality, we study the connections between the Racah algebra and the centralizers of tensor products of three (possibly different) irreducible representations of $\mathfrak{su}(2)$. As a first step we show that the Racah algebra always surjects onto the centralizer. We then offer a conjecture regarding the description of the kernel of the map, which depends on the irreducible representations. If true, this conjecture would provide a presentation of the centralizer as a quotient of the Racah algebra. We prove this conjecture in several cases. In particular, while doing so, we explicitly obtain the Temperley--Lieb algebra, the Brauer algebra and the one-boundary Temperley--Lieb algebra as quotients of the Racah algebra.
\end{abstract}

\maketitle

%

%

\vspace{3mm}

\section{Introduction}

The purpose of this paper is to make precise the connections between the Racah algebra, the centralizers of the Lie algebra $\mathfrak{su}(2)$
and some algebras connected to the braid group like the Temperley--Lieb or Brauer algebras.

On the one hand, the Racah algebra has been introduced in \cite{GZ} to study the Racah $W$-coefficients or $6j$-symbols \cite{Rac} 
intervening in the coupling of three angular momenta. These Racah $W$-coefficients are associated specifically to the direct sum 
decomposition of tensor product of three $\mathfrak{su}(2)$ representations. They are known to be expressible in terms of polynomials bearing the same name and 
sitting at the top of the discrete Askey scheme for $q=1$.
The Racah algebra has hence become a central tool to synthesize the (bispectral) properties of hypergeometric polynomials.

On the other hand, the Schur--Weyl duality provides also a way to study the decomposition of a tensor product of representation 
of $\mathfrak{su}(2)$ into a direct sum. Is is based on the study of the centralizer of the action of $\mathfrak{su}(2)$ on the tensor product.
For fundamental representations of $\mathfrak{su}(2)$, this centralizer is a quotient of the permutation group called the Temperley--Lieb algebra \cite{TL}
and for the irreducible representations of dimension $3$, it is the Brauer algebra \cite{Brauer,LZ,LZ2}.

In this paper, we offer a conjecture regarding the description in terms of generators and relations of the centralizer of the action of $\mathfrak{su}(2)$
on the tensor product of any three irreducible representations, characterized by three half-integers or integers, $j_1$, $j_2$ and $j_3$. 
Our starting point is that the Racah algebra surjects onto the centralizer, or in other words, that the ``intermediate'' Casimir elements (see Subsection \ref{sec:IC}) 
generate the whole centralizer of the tensor product of any three irreducible representations.
Then our strategy is to find a quotient of the Racah algebra such that it becomes isomorphic to the centralizer. Proofs will be given in a number of 
special cases to support our conjecture. In particular, the Temperley--Lieb and the Brauer
algebras are recovered through this approach in the cases involving respectively the product of three fundamental representations ($j_1=j_2=j_3=\frac{1}{2}$) 
and of three irreducible representations of dimension $3$
($j_1=j_2=j_3=1$), respectively. These results nicely establish the connection between the Racah, Temperley--Lieb and Brauer algebras.
We also prove our conjecture in the following instances: $j_1=j$, $j_2=j_3=\frac{1}{2}$ ($j=1,\frac{3}{2},\dots$) and $j_1=\frac{1}{2}$, $j_2=j_3=1$. 
For both cases, as far as we know, it is the first time that the centralizer is described explicitly in terms of generators and relations. 
Remarkably, in the first situation, we arrive at a specialization of an algebra previously studied  and known as the one-boundary Temperley-Lieb algebras \cite{MS, MW,NRG}. 
In addition, the conjecture is also verified for $j_1=j_2=j_3=\frac{3}{2}$, which is the first example not involving the spins $\frac{1}{2}$ and $1$. Finally, a proof of the conjecture is provided for $j_1=j$, $j_2=\frac{1}{2}$, $j_3=k$ (for any $j$ and $k$) which includes an infinite family of cases where the three spins are all distinct.

The plan of the paper is as follows. Section \ref{sec:gen} is concerned with the general situation. Relevant results about
the Lie algebra  $\mathfrak{su}(2)$ are recalled in Subsection \ref{sec:irreps} and intermediate Casimirs are discussed in Subsection \ref{sec:IC}.
Subsection \ref{sec:CR} presents the main conjecture which provides details on the connection between the centralizer and the Racah algebra. Subsection \ref{sec:sur} contains the proof 
that the image of the Racah algebra generates all the centralizer (this statement is sometimes called the 
first fundamental theorem of the invariant theory).  
Subsection \ref{sec:S3} is concerned with the study of the action of the permutation group of three elements on the quotiented Racah algebra.
Proofs of the conjecture for different particular cases are offered in the sections that follow. 
Section \ref{sec:TL} focuses on $j_1=j_2=j_3=\frac{1}{2}$ and provides the explicit connection 
between the Racah algebra and the Temperley--Lieb algebra (see Theorem \ref{th:TL}).
Section \ref{sec:brauer} examines the case $j_1=j_2=j_3=1$. In so doing, it is found that the quotiented Racah algebra is isomorphic to the Brauer algebra (see Theorem \ref{th:br}).
In Section \ref{sec:sff}, attention is paid to the case $j_1=1,\frac{3}{2},2,\dots$ and $j_2=j_3=\frac{1}{2}$ and by showing the conjecture to hold, we obtain 
a first description of the corresponding centralizer in terms of generators and relations. 
The case $j_1=\frac{1}{2}$ and $j_2=j_3=1$ is dealt within Section \ref{sec:f11}. We find a simple presentation 
of the centralizer that we propose to call the one-boundary Brauer algebra.
In Section \ref{sec:32}, we outline the proof of the 
conjecture when $j_1=j_2=j_3=\frac{3}{2}$. Finally, in Section \ref{sec:j12k}, the conjecture is proven 
for the case $(j_1,j_2,j_3)=(j,\frac{1}{2},k)$ for any $j$ and $k$.
The paper ends with concluding remarks.

\section{Racah algebra and centralizer for $\mathfrak{su}(2)$ \label{sec:gen}}

\subsection{Lie algebra  $\mathfrak{su}(2)$ and its irreducible representations\label{sec:irreps}}

The Lie algebra $\mathfrak{su}(2)$ is generated by $\mathfrak{s}^\alpha$ for $\alpha=1,2,3$ satisfying
\begin{equation}
 \left[\mathfrak{s}^\alpha\, ,\, \mathfrak{s}^\beta\right]=i\ \epsilon_{\alpha\beta\gamma}\   \mathfrak{s}^\gamma\ ,
\end{equation}
where $\epsilon_{\alpha\beta\gamma}$ is the Levi-Civita tensor. The quadratic Casimir of the universal enveloping algebra $U(\mathfrak{su}(2))$ is given by
\begin{equation}
\mathcal{K}=\sum_{\alpha=1}^3(\mathfrak{s}^\alpha)^2 \ .\label{eq:cas}
\end{equation}
We denote by $[2j]$ ($j=0,\frac{1}{2},1,\frac{3}{2},\dots$) the finite irreducible representation of $\mathfrak{su}(2)$ of dimension $2j+1$.
We recall that the image of the Casimir \eqref{eq:cas} in $\text{End}([2j])$ is $j(j+1)\un_{2j+1}$ where $\un_{2j+1}$ is the $2j+1$ by $2j+1$ identity matrix.
We also use the name spin-$j$ representation to refer to $[2j]$ and represent $[2j]$ by the Young tableau with $2j$ boxes 
 $\underbrace{\begin{tikzpicture}[scale=0.3]
\diag{0}{0}{2}
\end{tikzpicture} \cdots \begin{tikzpicture}[scale=0.3]
\diag{0}{0}{1}
\end{tikzpicture}}_{2j} $.

We choose three half-integers or integers $j_1$, $j_2$ and $j_3$ and define $\mathcal{J}_{ab}=\{  |j_a-j_b|,|j_a-j_b|+1,\dots, j_a+j_b\}$ for $1\leq a<b\leq 3$.
The tensor product of the two fold representation $[2j_a] \otimes  [2j_b]$ is reducible into the following direct sum 
\begin{equation}
 [2j_a] \otimes  [2j_b]= \bigoplus_{j \in \mathcal{J}_{ab}} [2j]\ .\label{eq:tp2}
\end{equation}
Similarly, the direct sum decomposition of the three fold tensor product $[2j_1] \otimes  [2j_2] \otimes  [2j_3]$ is given by
\begin{equation}
 [2j_1] \otimes  [2j_2] \otimes  [2j_3]=\bigoplus_{j \in \mathcal{J}_{123}}\ d_j [2j]\ ,\qquad\text{for}\quad d_j\in\mathbb{Z}_{>0} \quad \text{and}\qquad \mathcal{J}_{123} \subset \{0,\frac{1}{2},1,\dots\}\ . \label{eq:J4}
\end{equation}
To determine the set $\mathcal{J}_{123}$ and the degeneracy $d_j$ it is convenient to draw the Bratteli diagram $\cB(j_1,j_2,j_3)$.
That is a graph made out of three rows of vertices. The top row contains one vertex with the representation $[2j_1]$, 
the middle row contains the representations $[2k]$ with $k \in \mathcal{J}_{12}$ and the bottom row contains the representations $[2\ell]$ 
with $\ell \in \mathcal{J}_{123}$.
Edges are drawn between the vertex of the first row and all those of the second one.
An edge is also drawn between the vertex $[2k]$ of the second row and the vertex $[2\ell]$ of the third row if and only if the representation $[2\ell]$ is in the direct 
sum decomposition of $[2k] \otimes [2j_3]$ (see Figures \ref{fig:TL}-\ref{fig:32} for examples).
In other words,  $\cB(j_1,j_2,j_3)$ is the Bratteli diagram for the inclusion of centralizer algebras
\[\text{End}_{\mathfrak{su}(2)}([2j_1])\subset \text{End}_{\mathfrak{su}(2)}([2j_1]\otimes [2j_2])\subset \text{End}_{\mathfrak{su}(2)}([2j_1]\otimes [2j_2]\otimes [2j_3])\ ,\]
for the natural inclusions $a\mapsto a\otimes 1$.

For a given $\ell\in \mathcal{J}_{123}$, the degeneracy $d_\ell$ in formula \eqref{eq:J4} is the number of edges connected to the
vertex $[2\ell]$ in the third row of the Bratteli diagram $\cB(j_1,j_2,j_3)$. 
The construction of $\cB(j_a,j_b,j_c)$ (for $1\leq a,b,c \leq 3$ two by two different) allows us to define the following set
\begin{equation}
\mathcal{M}_{abc}= \{\ell(\ell+1)-k(k+1) \ | \ ( [2k], [2\ell])\text{ is an edge in }\cB(j_a,j_b,j_c)\text{ from the second to the third row} \}.
\end{equation}
Let us emphasize that there are no repeated numbers in these $\mathcal{M}_{abc}$.
We can also see that $\mathcal{M}_{123}=\mathcal{M}_{213}$. However, the set $\mathcal{M}_{231}$ is in general different of $\mathcal{M}_{123}$ 
(see Section \ref{sec:sff} or \ref{sec:f11} for such examples).

\subsection{Intermediate Casimirs\label{sec:IC}}

Let us fix again three half-integers or integers $j_1$, $j_2$ and $j_3$. We denote by $s^\alpha_a$ 
the image of $\mathfrak{s}^\alpha$ in $\text{End}([2j_a])$. 
For example, for $j_a=\frac{1}{2}$, one gets $s^\alpha_a=\frac{1}{2}\sigma^\alpha$ where $\sigma^\alpha$ are the Pauli matrices. 
Then, we can define an action of $\mathfrak{su}(2)$  in the space $[2j_1] \otimes [2j_2] \otimes [2j_3]$ by 
\begin{equation}
 \mathfrak{s}^\alpha \cdot ( v_1 \otimes v_2 \otimes v_3) =( s^\alpha_1 v_1) \otimes v_2 \otimes v_3 +v_1 \otimes ( s^\alpha_2 v_2) \otimes v_3 + v_1 \otimes v_2 \otimes (s^\alpha_3 v_3)\;,
\end{equation}
where $v_a\in [2j_a]$. By abuse of notation, we shall write in the following $s^\alpha_1$ for $s^\alpha_1 \otimes \un_{2j_2+1} \otimes \un_{2j_3+1}$ with analogous understanding of $s^\alpha_2$ and $s^\alpha_3$.

We define the following Casimirs
\begin{eqnarray}
 &&K_a=\sum_{\alpha=1}^3 ({s}^\alpha_a)^2 \quad\text{for}\quad a =1,2,3_ ,\label{eq:cas1}\\
 &&K_{ab}=\sum_{\alpha=1}^3({s}^\alpha_a+{s}^\alpha_b)^2 \quad\text{for}\quad 1\leq a<b\leq 3\ ,\\
 &&{K}_{123}=\sum_{\alpha=1}^3({s}^\alpha_1+{s}^\alpha_2+{s}^\alpha_3)^2\ .\label{eq:cas3}
\end{eqnarray}
They satisfy the following equality:
\begin{equation}\label{eq:CC}
{K}_1+{K}_2+{K}_3+{K}_{123}={K}_{12}+{K}_{23}+{K}_{13}\;.
\end{equation}
Since $[2j_a]$ is irreducible, one gets $K_a=j_a(j_a+1)\un_{2j_a+1}$. 
The intermediate Casimir $K_{ab}$ is diagonalizable and its eigenvalues are $j(j+1)$ for $j \in \mathcal{J}_{ab}$.
We have
\begin{equation}
 \prod_{j\in \mathcal{J}_{ab}} ({K}_{ab}-j(j+1))=0\ .\label{eq:CP1}
\end{equation}
Similarly, the Casimir $K_{123}$ satisfies
\begin{equation}
 \prod_{j\in \mathcal{J}_{123}} ({K}_{123}-j(j+1))=0\ .\label{eq:CP2}
\end{equation}

Since $K_{12}$ commutes with $K_{123}$, we can diagonalize them in the same basis and it follows that the spectrum
of ${K}_{123}-K_{12} $ is included into the set 
\begin{equation}
\left\{\ell(\ell+1)-k(k+1) \ \Big|\  \ell\in \mathcal{J}_{123}, k\in\mathcal{J}_{12}   \right\}  \ .              
\end{equation}
In fact, we can make a more precise statement. The eigenvalues $\ell(\ell+1)$ for $K_{123}$ and $k(k+1)$ for $K_{12}$ are obtained simultaneously (that is, the eigenspaces have a non-trivial intersection) if and only if the corresponding vertices in the Bratteli diagram $\cB(j_1,j_2,j_3)$ are connected. Therefore, the spectrum of ${K}_{123}-K_{12} $ is the set
\begin{equation}
 \left\{ m\ \Big|\  m\in \mathcal{M}_{123}   \right\}\ .
\end{equation}
Similarly, the spectrum of $K_{123}-K_{23}$ is 
$
 \left\{ m\ \Big|\  m\in \mathcal{M}_{231}   \right\} 
$ and the spectrum of $K_{123}-K_{13}$ is 
$
 \left\{ m\ \Big|\  m\in \mathcal{M}_{132}   \right\} 
$.
From relation \eqref{eq:CC}, 
we see that ${K}_{12}+{K}_{23}={K}_{123}-K_{13}+{K}_1+{K}_2+{K}_3 $ and thus that the spectrum of ${K}_{12}+{K}_{23}$ is 
\begin{equation}
\left\{m+\sum_{a=1}^3 j_a(j_a+1) \ \Big|\ m\in \mathcal{M}_{132}  \right\}  \ .              
\end{equation}
We obtain from these spectra the minimal polynomials of ${K}_{12}+{K}_{23}$, $K_{123}-K_{12}$ and $K_{123}-K_{23}$.

\subsection{Centralizer and Racah algebra \label{sec:CR}}

The centralizer ${\mathcal{C}}_{j_1j_2j_3}$ of the image of $\mathfrak{su}(2)$ in $\text{End}([2j_1] \otimes [2j_2] \otimes [2j_3])$ is defined by
\begin{eqnarray}
{\mathcal{C}}_{j_1j_2j_3}&=&\text{End}_{\mathfrak{su}(2)}( [2j_1] \otimes [2j_2] \otimes [2j_3] )\\
 &=&\{\ M\in \text{End}([2j_1] \otimes [2j_2] \otimes [2j_3] )  \ |\ [M, s^\alpha_1+s^\alpha_2+s^\alpha_3]=0, \alpha=1,2,3\ \}\ .
\end{eqnarray}
We recall that the decomposition \eqref{eq:J4} allows one to compute the dimension of the centralizer ${\mathcal{C}}_{j_1j_2j_3}$:
\begin{equation}
\text{dim}\left( {\mathcal{C}}_{j_1j_2j_3}\right)= \sum_{j\in \mathcal{J}_{123}} d_j^2\ . \label{eq:dim}
\end{equation}
As explained in the introduction, the goal of this paper is to give a description of ${\mathcal{C}}_{j_1j_2j_3}$ in terms of generators and defining relations.
We now introduce the Racah algebra for this purpose.
\begin{definition} The universal Racah algebra $\mathcal{R}(\alpha_1,\alpha_2,\alpha_3)$ is generated by $ A $, $ B $ and central elements $\alpha_1,\alpha_2,\alpha_3$ and $ C $
subject to the following defining relations
 \begin{eqnarray}
[ B ,[ A , B ]]&=&-2 B ^2-2\{ A , B \}
     +2 ( C +\alpha_1+\alpha_2+\alpha_3) B +2(\alpha_1- C )(\alpha_3-\alpha_2)\ ,\label{eq:UR1}\\
{} [ A ,[ B , A ]]&=&-2 A ^2-2\{ A , B \}
     +2( C +\alpha_1+\alpha_2+\alpha_3) A  +2(\alpha_1-\alpha_2)(\alpha_3- C )\ ,\label{eq:UR2}
\end{eqnarray}
where $\{\cdot,\cdot\}$ is the anticommutator.
\end{definition}
We use the notation $\mathcal{R}(\alpha_1,\alpha_2,\alpha_3)$ indicating the central elements $\alpha_i$ for later convenience when we will replace these central elements by numbers.

The connection between the Racah algebra and the centralizer is given in the following known proposition
\begin{prop}\cite{GZ}\label{pro:phimor}
 The map $\phi$ from $\mathcal{R}(\alpha_1,\alpha_2,\alpha_3)$ to ${\mathcal{C}}_{j_1j_2j_3}$ defined by
 \begin{eqnarray}
  \phi(\alpha_i)=j_i(j_i+1)\quad,\qquad \phi( A )=K_{12}\quad,\qquad \phi( B )=K_{23}\quad\text{and}\qquad  \phi( C )=K_{123}\label{eq:AK}
 \end{eqnarray}
is well-defined (\textit{i.e.} $K_{12}$, $K_{23}$, $K_{123}$ are in ${\mathcal{C}}_{j_1j_2j_3}$) and is an algebra homomorphism.
\end{prop}
The surjectivity of the map $\phi$ is proven in the Corollary \ref{cor:sur} below. However this map is not injective. 
Indeed, for example $\prod_{j\in \mathcal{J}_{12}} (A-j(j+1))$ is in the kernel of $\phi$ because of \eqref{eq:CP1}.
The main conjecture of this paper consists in finding a quotient of the universal Racah algebra such that $\phi$ becomes a bijection.
\begin{conj}\label{conj:t} Let $j_1$, $j_2$ and $j_3$ be three positive half-integers or integers 
and let the sets $\mathcal{J}$ and ${\mathcal{M}}$ be defined as in Section \ref{sec:irreps}.
The quotient $\overline{\mathcal{R}}(\alpha_1,\alpha_2,\alpha_3)$ of the universal Racah algebra $\mathcal{R}(\alpha_1,\alpha_2,\alpha_3)$ by the following relations
\begin{eqnarray}
&& \alpha_i=j_i(j_i+1)\ ,\label{eq:quo0}\\
&&\prod_{j\in \mathcal{J}_{12}} ( A -j(j+1))=0\ ,\quad \prod_{j\in \mathcal{J}_{23}} ( B -j(j+1))=0\ ,\quad 
\prod_{j\in \mathcal{J}_{123}}  ( C -j(j+1))=0\label{eq:quo1}\\
&&\prod_{j\in \mathcal{J}_{13}} \big( C - A - B +\alpha_1+\alpha_2+\alpha_3-j(j+1)\big)=0\ ,\label{eq:quo2}\\
&&\prod_{m\in \mathcal{M}_{132}}\big( A + B -\alpha_1-\alpha_2-\alpha_3-m\big)=0\ ,\label{eq:quo3}\\
&&\prod_{m\in \mathcal{M}_{123}}\big( C - A -m\big)=0\ , \qquad \prod_{m\in\mathcal{M}_{231}}\big( C - B -m\big)=0\ ,\label{eq:quo5}
\end{eqnarray}
is isomorphic to ${\mathcal{C}}_{j_1j_2j_3}$.
The isomorphism $\overline{\phi}$ is given by $\overline{\phi}( A )=K_{12}$, $\overline{\phi}( B )=K_{23}$ and $\overline{\phi}( C )=K_{123}$.
\end{conj}
Let us emphasize that if the conjecture is true, relations \eqref{eq:UR1}-\eqref{eq:UR2} with \eqref{eq:quo1}-\eqref{eq:quo5} provide a presentation in terms 
of generators and relations of the centralizer ${\mathcal{C}}_{j_1j_2j_3}$.
Thanks to the results of Section \ref{sec:IC} , it is easy to show that $\overline{\phi}$ is an homomorphism from $\overline{\mathcal{R}}(\alpha_1,\alpha_2,\alpha_3)$ to ${\mathcal{C}}_{j_1j_2j_3}$ 
(to prove the homomorphism for relation \eqref{eq:quo2}, we have used relation \eqref{eq:CC}).
It remains to prove that $\overline{\phi}$ is injective. We did not succeed in finding a general proof: to that end one should show that the relations \eqref{eq:quo1}-\eqref{eq:quo5} 
generate the whole kernel of $\phi$.
We shall however prove this conjecture for a number of particular values of $j_1$, $j_2$ and $j_3$.

Let us remark that the case of the centralizer of the two-fold tensor product is much simpler. 
The direct sum decomposition \eqref{eq:tp2} of $[2j_1] \otimes [2j_2]$ is multiplicity free. Therefore, the centralizer $\text{End}_{\mathfrak{su}(2)}( [2j_1] \otimes [2j_2] )$ is an abelian algebra of dimension 
$\text{dim}(\mathcal{J}_{12})$ and is isomorphic to the algebra generated by one generator ${A}$ subject to $\prod_{j\in \mathcal{J}_{12}}(A-j(j+1))=0$.

\subsection{Surjectivity from $\mathcal{R}(\alpha_1,\alpha_2,\alpha_3)$ to ${\mathcal{C}}_{j_1j_2j_3}$ \label{sec:sur}}

\begin{prop}\label{cor:sur}
 The homomorphism $\phi$ defined in Proposition \ref{pro:phimor} is surjective.
\end{prop}
\proof Let $j\in \mathcal{J}_{123}$. In the space $[2j_1] \otimes [2j_2] \otimes [2j_3]$, there exists a subspace of dimension $d_j$ stable under the action of the centralizer.
By diagonalizing $K_{12}$, we can choose in this subspace $d_j$ independent vectors $v_p$ ($p=p_{min},p_{min}+1,\dots p_{max}$ and $p_{max}-p_{min}+1=d_j$) such that
\begin{equation}
 K_{12} v_p=p(p+1)v_p\quad\text{,}\qquad  K_{123}v_p=j(j+1)v_p\ .
\end{equation}
The numbers $p_{min}$ and $p_{max}$ depend on the choices of $j$ and also of $j_1$, $j_2$ and $j_3$. They can be read directly on the Bratteli diagram $\mathcal{B}(j_1,j_2,j_3)$: 
they correspond to the rightmost and leftmost vertices on the second line connected to the vertex $[2j]$ of the third line.
Following the results in \cite{Z91}, we know that $K_{23}$ acts trigonally on the vectors $v_p$:
\begin{equation}
 K_{23}v_p=a_{p+1} v_{p+1}+a_p v_{p-1} + b_p v_p\ .
\end{equation}
The explicit values of $a_p$ can be computed by using the commutation relations of the Racah algebra \cite{Z91} and we notice that $a_p\neq 0$ for $p=p_{min}+1,p_{min}+2,\dots,p_{max}$. 
Therefore the subspace $\text{span}(\{v_p\})$ of $[2j_1] \otimes [2j_2] \otimes [2j_3]$ is a finite irreducible representation of dimension $d_j$ for the Racah algebra. This result holds 
for any $j\in \mathcal{J}_{123}$. Therefore the dimension of the image of $\phi$ is at least $\text{dim}\left( {\mathcal{C}}_{j_1j_2j_3}\right)= \sum_{j\in \mathcal{J}_{123}} d_j^2$ which 
proves the surjectivity of $\phi$.
 \endproof

\begin{rem}
As explained previously, relations \eqref{eq:quo0}-\eqref{eq:quo5} are in the kernel of $\phi$. Therefore, Proposition \ref{cor:sur} implies also that $\overline{\phi}$ defined in the Conjecture \ref{conj:t} 
is surjective.
\end{rem}

To prove Proposition \ref{cor:sur}, we showed that the elements $K_{12},K_{23}$ and $K_{123}$ generate the whole centralizer of the diagonal action of $U(sl_2)$ in the three-fold tensor product of three representations.
Let us remark that there exists a similar statement at the algebraic level. The algebraic intermediate Casimirs are defined by replacing  $s$ by $\mathfrak{s}$ in relations \eqref{eq:cas1}--\eqref{eq:cas3}.
These algebraic Casimirs generate the whole centralizer of the diagonal embedding of $U(sl_2)$ in $U(sl_2)^{\otimes 3}$. An analogous result (using the whole set of intermediate Casimir elements) holds also for the centralizer of the diagonal embedding of $U(sl_2)$ in $U(sl_2)^{\otimes n}$. This can be checked directly by algebraic manipulations, starting with an arbitrary element of $U(sl_2)^{\otimes n}$ written in a PBW basis and using the conditions that it commutes with the diagonal embedding of 
$U(sl_2)$.

\subsection{Invariance under permutation of $\{\alpha_1,\alpha_2,\alpha_3\}$ \label{sec:S3}}
We prove a general result which shows that it suffices to check  the conjecture only once for each orbit under permutations of $j_1,j_2,j_3$. In other words, when verifying the conjecture, we are always allowed to reorder the three spins $j_1,j_2,j_3$ as we want. For example, in Section \ref{sec:j12k}, we will choose to order $j_1,j_2,j_3$ such that $j_1\geq j_3\geq j_2$ (in general, putting the smallest spin in the middle minimizes the degree of the characteristic equations of $A$ and $B$).
\begin{prop} 
Let $j_1$, $j_2$ and $j_3$ be three positive half-integers or integers. If Conjecture \ref{conj:t} is true for the sequence of spins $(j_1,j_2,j_3)$ then it is also true for every permutation of $j_1,j_2,j_3$.
\end{prop}
\begin{proof}
Let $j_1$, $j_2$ and $j_3$ be three positive half-integers or integers. 
For any two representations $V$ and $W$ of $\mathfrak{su}(2)$, the representations $V\otimes W$ and $W\otimes V$ are isomorphic. 
Therefore, for any permutation $\pi$ of $\{1,2,3\}$, the centralizer $\mathcal{C}_{j_1j_2j_3}$ is isomorphic to $\mathcal{C}_{j_{\pi(1)}j_{\pi(2)}j_{\pi(3)}}$. 
So in order to prove the proposition, we need to prove that the quotiented Racah algebra $\overline{\mathcal{R}}(\alpha_{\pi(1)},\alpha_{\pi(2)},\alpha_{\pi(3)})$ 
is isomorphic to $\overline{\mathcal{R}}(\alpha_{1},\alpha_{2},\alpha_{3})$ for any permutation $\pi$ of $\{1,2,3\}$.

We claim that the following maps on generators provide isomorphisms of algebras:
\begin{equation}\label{iso13}
\phi_1\ :\ \ \ \begin{array}{rcl}
\overline{\mathcal{R}}(\alpha_{3},\alpha_{2},\alpha_{1}) & \to & \overline{\mathcal{R}}(\alpha_{1},\alpha_{2},\alpha_{3})\\
A & \mapsto & B\\
B & \mapsto & A\\
C & \mapsto & C
\end{array}
\end{equation}
\begin{equation}\label{iso12}
\phi_2\ :\ \ \ \begin{array}{rcl}
\overline{\mathcal{R}}(\alpha_{2},\alpha_{1},\alpha_{3}) & \to & \overline{\mathcal{R}}(\alpha_{1},\alpha_{2},\alpha_{3})\\
A & \mapsto & A\\
B & \mapsto & C+\alpha_1+\alpha_2+\alpha_3-A-B\\
C & \mapsto & C
\end{array}
\end{equation}
The maps are obviously invertible so it remains to prove that they extend to homomorphisms of algebras, by checking that they preserve the defining relations of the quotiented Racah algebra, namely relations \eqref{eq:UR1}--\eqref{eq:UR2} and \eqref{eq:quo1}--\eqref{eq:quo5}.

For $\phi_1$, the verification of relations \eqref{eq:UR1}--\eqref{eq:UR2} is immediate. For the characteristic equations \eqref{eq:quo1}--\eqref{eq:quo5}, the verification follows immediately once we know how the sets $\mathcal{J}$ and $\mathcal{M}$ involved in these relations transform when exchanging $1$ and $3$. We see at once that $\mathcal{J}_{123}$, $\mathcal{J}_{13}$ and $\mathcal{M}_{132}$ are invariant while $\mathcal{J}_{12}$ and $\mathcal{J}_{23}$ are exchanged, and so are $\mathcal{M}_{123}$ and $\mathcal{M}_{213}$. We used that $\mathcal{J}_{ab}=\mathcal{J}_{ba}$ and  $\mathcal{M}_{abc}=\mathcal{M}_{bac}$.

For $\phi_2$, the fact that relations \eqref{eq:UR1}--\eqref{eq:UR2} hold can be checked by a straightforward explicit calculation. For the characteristic equations \eqref{eq:quo1}--\eqref{eq:quo5}, as before, the verification follows immediately once we know how the sets $\mathcal{J}$ and $\mathcal{M}$ transform when exchanging $1$ and $2$. We see at once that $\mathcal{J}_{123}$, $\mathcal{J}_{12}$ and $\mathcal{M}_{123}$ are invariant while $\mathcal{J}_{23}$ and $\mathcal{J}_{13}$ are exchanged, and so are $\mathcal{M}_{132}$ and $\mathcal{M}_{231}$.

To conclude the proof, we note that for any permutation $\pi$ of $\{1,2,3\}$, we obtain an isomorphism between $\overline{\mathcal{R}}(\alpha_{\pi(1)},\alpha_{\pi(2)},\alpha_{\pi(3)})$ and $\overline{\mathcal{R}}(\alpha_{1},\alpha_{2},\alpha_{3})$ by a suitable composition of $\phi_1$ and $\phi_2$ (since the transpositions $(1,2)$ and $(1,3)$ generate the whole symmetric group on $3$ letters).
\end{proof}

\section{Quotiented Racah $\overline{\mathcal{R}}\left(\frac{3}{4},\frac{3}{4},\frac{3}{4}\right)$ and Temperley-Lieb algebra\label{sec:TL}}

In this section, we are interested in studying in detail the centralizer  ${\mathcal{C}}_{j_1j_2j_3}$ for the case $j_1=j_2=j_3=\frac{1}{2}$ and to prove that it is isomorphic to
 $\overline{\mathcal{R}}\left(\frac{3}{4},\frac{3}{4},\frac{3}{4}\right)$. 
 We know that a specialization of the Temperley--Lieb algebra is the centralizer and we give an explicit isomorphism 
 between this Temperley--Lieb algebra and $\overline{\mathcal{R}}\left(\frac{3}{4},\frac{3}{4},\frac{3}{4}\right)$.
 
The Bratteli diagram  $\mathcal{B}(\frac{1}{2},\frac{1}{2},\frac{1}{2})$ is displayed in Figure \ref{fig:TL}.
We deduce that $\text{dim}\left( {\mathcal{C}}_{\frac{1}{2},\frac{1}{2},\frac{1}{2}}\right)= 5$, 
$\mathcal{J}_{12}=\mathcal{J}_{13}=\mathcal{J}_{23}=\{0,1\}$, $\mathcal{J}_{123}=\{\frac{1}{2},\frac{3}{2}\}$ and
$\mathcal{M}_{123}=\mathcal{M}_{231}=\mathcal{M}_{132}=\{\frac{7}{4},-\frac{5}{4},\frac{3}{4}\}$.
\begin{figure}[htb]
\begin{center}
 \begin{tikzpicture}[scale=0.35]
\diag{0}{0}{1};\node at (2,-0.5) {$\frac{3}{4}$};

\draw (-0.5,-1.5) -- (-3,-3.5);\draw (1.5,-1.5) -- (4.3,-3.8);

\diag{-4}{-4}{2};\node at (-1,-4.5) {$2$};\filldraw (5,-4.5) circle (2mm);  \node at (6,-4.5) {$0$};

\draw (-3,-5.5) -- (-3,-8.5);\draw (-2,-5.5) -- (4,-8.5);\draw (5,-5.2) -- (5,-8.5);

\diag{-4.5}{-9}{3};\node at (-0.5,-9.5) {$\frac{15}{4}$};\diag{4.5}{-9}{1};\node at (6.5,-9.5) {$\frac{3}{4}$};

\node at (-14,-2.5) {$\otimes$};\diag{-13}{-2}{1};\node at (-14,-6.5) {$\otimes$};\diag{-13}{-6}{1};

\end{tikzpicture}
\caption{Bratteli diagram $\mathcal{B}(\frac{1}{2},\frac{1}{2},\frac{1}{2})$. On the right of each Young tableau, 
the corresponding value of the Casimir is recalled.\label{fig:TL}}
\end{center}
\end{figure}
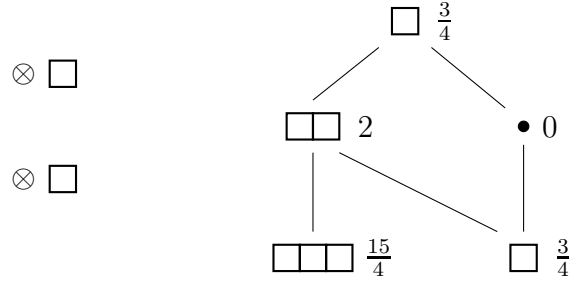

\begin{prop}\label{pro:1TL}
 The quotiented Racah algebra $\overline{\mathcal{R}}\left(\frac{3}{4},\frac{3}{4},\frac{3}{4}\right)$ is generated by the 
 central element $G$ and two generators $A$ and $B$ satisfying
  \begin{eqnarray}
 && A^2=2A\ ,\qquad B^2=2B\ ,\qquad (G-1)(G-4)=0\ ,\qquad G=\{A,B\}-2A-2B+4\ ,\label{eq:tl2}\qquad\\
 &&  GA=2\{A,B\}-3A-4B+6\ ,\qquad GB=2\{A,B\}-4A-3B+6\ ,\label{eq:lem03}\\
&& ABA=GA\ ,\qquad  BAB=GB\ ,\label{eq:tl1}
 \end{eqnarray}
 with the identification ${G}= C +\frac{1}{4}$.
\end{prop}
\proof
By using the sets $\mathcal{J}$ and $\mathcal{M}$ given at the beginning of this section and the identification given in the proposition, relations \eqref{eq:UR1}-\eqref{eq:UR2} 
and \eqref{eq:quo1}--\eqref{eq:quo5} become
\begin{eqnarray}
&& [B,[A,B]]=-2B^2-2\{A,B\}+2B(G+2)\quad,\quad   [[A,B],A]=-2A^2-2\{A,B\}+2A(G+2)\ ,\label{eq:quTL0}\\
 &&A(A-2)=0\quad,\quad B(B-2)=0\quad,\quad \left(G-1\right)\left(G-4\right)=0\ ,\label{eq:quTL1}\\
 &&  \left(G-A-B+2\right)\left(G-A-B\right)=0 \quad,\quad\left(A+B-4\right)\left(A+B-3\right)\left(A+B-1\right)=0\ ,\qquad\label{eq:quTL2}\\
  && \left(G-A-2\right)\left(G-A-1\right)\left(G-A+1\right)=0\quad , \quad
  \left(G-B-2\right)\left(G-B-1\right)\left(G-B+1\right)=0\ .\label{eq:quTL3}
\end{eqnarray}
By expanding the products in \eqref{eq:quTL0}--\eqref{eq:quTL3}, we can show that they are equivalent to the relations given in the proposition.
\endproof
The presentation of the quotient of the Racah algebra $\overline{\mathcal{R}}\left(\frac{3}{4},\frac{3}{4},\frac{3}{4}\right)$ can be simplified further.
Indeed, the fourth relation in \eqref{eq:tl2} allows to suppress the generator $G$ in the presentation.
\begin{prop}\label{pro:2TL}The quotiented Racah algebra $\overline{\mathcal{R}}\left(\frac{3}{4},\frac{3}{4},\frac{3}{4}\right)$ is generated by $A$ and $B$ subject to 
 \begin{eqnarray}
 &&A^2=2A\ ,\qquad B^2=2B\ ,\label{eq:tl4}\\
&&ABA=2\{A,B\}-3A-4B+6\ ,\qquad   BAB=2\{A,B\}-4A-3B+6 \ .\label{eq:tl5}
 \end{eqnarray}
\end{prop}
\proof  We must prove that the set of defining relations of Proposition \ref{pro:1TL} is equivalent to the ones of Proposition \ref{pro:2TL}.
To prove that the defining relations of Proposition \ref{pro:1TL} imply the ones of Proposition \ref{pro:2TL} is straightforward. 

To prove the implication in the other direction, we suppose that $A$ and $B$ satisfy relations \eqref{eq:tl4}-\eqref{eq:tl5} and we set 
$G=\{A,B\}-2A-2B+4$. Then, by multiplying this definition of $G$ by $A$ and $B$ on  the right and on the left, we get that $G$ commutes 
with $A$ and $B$ and that relations \eqref{eq:tl1} are satisfied. By using relations \eqref{eq:tl5}, one proves that relations \eqref{eq:lem03} also hold.
The proof of the third relation of \eqref{eq:tl2} is more involved
\begin{eqnarray}
 (G-1)(G-4)&=&(AB+BA-2A-2B+3)(AB+BA-2A-2B)\\
 &=&ABAB+BABA-2ABA-2BAB-AB-BA+2A+2B\ .\label{eq:re1}
\end{eqnarray}
We have used relations \eqref{eq:tl4} to prove \eqref{eq:re1}. Then, by multiplying on the left by $B$ the first relation of \eqref{eq:tl5} 
and by $A$ the second relation of \eqref{eq:tl5}, we get
expressions for $BABA$ and $ABAB$ and we prove that \eqref{eq:re1} vanishes which concludes the proof.
\endproof

A direct consequence of this proposition is that the dimension of the quotiented Racah algebra 
$\overline{\mathcal{R}}\left(\frac{3}{4},\frac{3}{4},\frac{3}{4}\right)$ is equal to 5. A basis is $\{1,A,B,AB,BA\}$. 
Indeed, it is straightforward to show that this set of elements is a generating family. The linear independence of
these elements is proven by noticing that their images by the natural map $\phi$ given in (\ref{eq:AK}) in $\text{End}([1]\otimes[1]\otimes[1])$ are linearly independent.

We have shown that the quotiented Racah algebra and the centralizer have the same dimension, and moreover that the map $\overline{\phi}$ in Conjecture \ref{conj:t} is injective (alternatively, we already know that the map $\overline{\phi}$ is surjective from Proposition \ref{cor:sur}). We conclude

\begin{thm}
Conjecture \ref{conj:t} is verified for $j_1=j_2=j_3=\frac{1}{2}$.
\end{thm}

\subsection{Connections with the Temperley--Lieb algebra}

It is however well-known that the centralizer of the action of $U(\mathfrak{su}(2))$ on the tensor product of the fundamental representation is a special case of the Temperley--Lieb algebra \cite{Jimbo}.
Therefore, in view of the preceding result, the quotiented Racah algebra $\overline{\mathcal{R}}\left(\frac{3}{4},\frac{3}{4},\frac{3}{4}\right)$ must be isomorphic to this algebra.
Let us recall the definition of the Temperley--Lieb algebra \cite{TL}.
\begin{definition}
The Temperley--Lieb algebra $TL_3(q)$ is generated by $\sigma_1$ and $\sigma_2$ with the following defining relations
\begin{eqnarray}
&& \sigma_1^2=(q+q^{-1})\sigma_1\quad,\quad \sigma_2^2=(q+q^{-1})\sigma_2\ ,\label{eq:TL1}\\
&& \sigma_1\sigma_2\sigma_1=\sigma_1\quad,\quad \sigma_2\sigma_1\sigma_2=\sigma_2\ .\label{eq:TL2}
\end{eqnarray}
\end{definition}
We can now state a theorem that clarifies the link between the Temperley--Lieb algebra and the Racah algebra.
\begin{thm}\label{th:TL}
 The quotiented Racah algebra $\overline{\mathcal{R}}\left(\frac{3}{4},\frac{3}{4},\frac{3}{4}\right)$ is isomorphic to the Temperley--Lieb algebra $TL_3(1)$. 
 This isomorphism is given explicitly by
\begin{eqnarray}
 \overline{\mathcal{R}}\left(\frac{3}{4},\frac{3}{4},\frac{3}{4}\right) & \to& TL_3(1)\\
A &\mapsto& 2-\sigma_1\nonumber\\
B &\mapsto& 2-\sigma_2.\nonumber
\end{eqnarray}
\end{thm}
\proof 
It is straightforward to prove that the relations of the Temperley--Lieb algebra  \eqref{eq:TL1}-\eqref{eq:TL2} are equivalent to the ones of the 
Racah algebra $\overline{\mathcal{R}}\left(\frac{3}{4},\frac{3}{4},\frac{3}{4}\right)$ given by \eqref{eq:tl4}-\eqref{eq:tl5}.
\endproof

The image of the central element $G=C+\frac{1}{4}=\{A,B\}-2A-2B+4$ calculated explicitly in $TL_3(1)$ simplifies to:
\[G\mapsto \sigma_1\sigma_2+\sigma_2\sigma_1-2\sigma_1-2\sigma_2+4\ .\] 
It is easy to recognize that the image of $G$ is indeed equal to $4P_1+P_2$ where $P_1$ (respectively, $P_2$) is the projector associated to the irreducible representation of $TL_3(1)$ of dimension 1 (respectively, of dimension 2).

Let us remark that Proposition \ref{pro:1TL} allows to decompose the quotiented Racah algebra according to the eigenvalues of $G$.
Indeed, $\overline{\mathcal{R}}\left(\frac{3}{4},\frac{3}{4},\frac{3}{4}\right)$ is the direct sum of the algebras:
\begin{itemize}
 \item $G=4$, $A=2$ and $B=2$\ ;
 \item $G=1$, $A^2=2A$, $B^2=2B$, $ABA=A$, $BAB=B$ and $\{A,B\}=2A+2B-3$.
\end{itemize}
The first is of dimension 1 and the second of dimension 4. This provides another way to find that 
$\text{dim}\left(\overline{\mathcal{R}}\left(\frac{3}{4},\frac{3}{4},\frac{3}{4}\right)\right)=5$.

\section{Quotiented Racah $\overline{\mathcal{R}}\left(2,2,2\right)$ and Brauer algebra \label{sec:brauer}}

In this section, we bring our attention to the case $j_1=j_2=j_3=1$ which corresponds to the centralizer ${\mathcal{C}}_{111}$. We prove that it is isomorphic to
 $\overline{\mathcal{R}}\left(2,2,2\right)$ and to the Brauer algebra.
 
The Bratteli diagram associated to the tensor product of three spin-1 representations is displayed in Figure \ref{fig:br}.
We observe that $\text{dim}({\mathcal{C}}_{111})=1^2+2^2+3^2+1^2=15$, $\mathcal{J}_{12}=\mathcal{J}_{13}=\mathcal{J}_{23}=\{0,1,3\}$, $\mathcal{J}_{123}=\{0,1,3,6\}$ and 
$\mathcal{M}_{123}=\mathcal{M}_{231}=\mathcal{M}_{132}=\{-4,-2,0,2,4,6\}$.

\begin{figure}[htb]
\begin{center}
 \begin{tikzpicture}[scale=0.35]
\diag{-0.5}{0}{2};\node at (2.5,-0.5) {$2$};
\draw (-1,-1.5) -- (-6,-3.5);\draw (0.5,-1.5)--(0.5,-3.5);\draw (2,-1.5) -- (6.7,-3.7);
\diag{-9.5}{-4}{4};\node at (-4.5,-4.5) {$6$};
\diag{-0.5}{-4}{2};\node at (2.5,-4.5) {$2$};
\filldraw (7.5,-4.5) circle (2mm); ; \node at (8.5,-4.5) {$0$};

\draw (-8.5,-5.5) -- (-14,-8.5); \draw (-6.5,-5.5) -- (-5,-8.5);  \draw (-4.5,-5.5) -- (3,-8.5);   

\draw (-0.5,-5.5) -- (-3,-8.5);\draw (0.5,-5.5) -- (4,-8.5);\draw (1.5,-5.5) -- (10.2,-8.8);

\draw (7.2,-5.2) -- (5,-8.5);

\diag{-17}{-9}{6};\node at (-10,-9.5) {$12$};
\diag{-6}{-9}{4};\node at (-1,-9.5) {$6$};
\diag{3}{-9}{2};\node at (6,-9.5) {$2$};
\filldraw (11,-9.5) circle (2mm);\node at (12,-9.5) {$0$};

\node at (-24,-2.5) {$\otimes$};\diag{-23}{-2}{2};
\node at (-24,-7) {$\otimes$};\diag{-23}{-6.5}{2};
\end{tikzpicture}
\caption{Bratteli diagram $\cB(1,1,1)$. 
The value of the Casimir is given on the right of each corresponding Young tableau.\label{fig:br}  }
\end{center}
\end{figure}
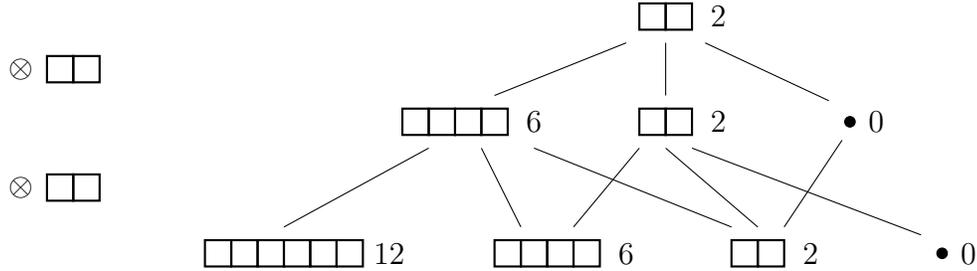

From the definition of the sets $\mathcal{J}$ and $\mathcal{M}$ given previously, we see that the relations \eqref{eq:UR1}-\eqref{eq:UR2} and \eqref{eq:quo1}-\eqref{eq:quo5} of the 
quotiented Racah $\overline{\mathcal{R}}\left(2,2,2\right)$ become
\begin{eqnarray}
&& 2BAB=AB^2+B^2A-2B^2-2\{A,B\}+12B+2CB\ ,\label{eq:br6}\\
&& 2ABA=BA^2+A^2B-2A^2-2\{A,B\}+12A+2CA\ ,\label{eq:br7}\\
&& A(A-2)(A-6)=0\ ,\qquad B(B-2)(B-6)=0\ ,\qquad C(C-2)(C-6)(C-12)=0\ ,\label{eq:br1}\\
&& (C-A-B+6)(C-A-B+4)(C-A-B)=0\ ,\label{eq:br2}\\
&& (A+B-12)(A+B-10)(A+B-8)(A+B-6)(A+B-4)(A+B-2)=0\ ,\label{eq:br3}\\
&& (C-A-6)(C-A-4)(C-A-2)(C-A)(C-A+2)(C-A+4)=0\ ,\label{eq:br4}\\
&& (C-B-6)(C-B-4)(C-B-2)(C-B)(C-B+2)(C-B+4)=0\ .\label{eq:br5}
\end{eqnarray}

We want to demonstrate that $\text{dim}(\overline{\mathcal{R}}\left(2,2,2\right))=15$. Let
\begin{equation}
 \mathcal{S}=\{ 1, A,B,A^2,B^2,AB,BA,A^2B,AB^2,ABA,BAB,BA^2,BABA,A^2B^2,ABAB\}\ .
\end{equation}
We can show after some algebraic manipulations that 
\begin{equation}
 \mathcal{S}_r=\mathcal{S} \cup C\mathcal{S} \cup C^2\mathcal{S} \cup C^3\mathcal{S}
\end{equation}
is a generating set. Therefore we can construct the 60 by 60 matrices $A_r$, $B_r$ and $C_r$ corresponding to the regular actions of $A$, $B$ and $C$ on the set $\mathcal{S}_r$.
By asking that $A_r$, $B_r$ and $C_r$ satisfy \eqref{eq:br6}-\eqref{eq:br5}, we find $40$ independent constraints between the elements of $\mathcal{S}_r$.
It follows that we can reduce the generating set $\mathcal{S}_r$ to
\begin{equation}
 \mathcal{S}'_r=\mathcal{S} \cup \{ C,CA^2,CB^2,CA^2B,CAB^2 \}\ .
\end{equation}
We now repeat the same procedure to construct at this point
the 20 by 20 matrices corresponding to the regular actions on $\mathcal{S}'_r$ and ask that they satisfy the relations of $\overline{\mathcal{R}}\left(2,2,2\right)$.
We find 5 supplementary independent relations which allow to reduce the generating set to $\mathcal{S}$. 

Moreover, the images of the elements of $\mathcal{S}$ in $\text{End}([2]\otimes [2] \otimes [2])$ given by the homomorphism $\overline \phi$ (see Conjecture \ref{conj:t}) are 15 linearly independent matrices. We conclude that $\mathcal{S}$ is a basis and 
$\text{dim}(\overline{\mathcal{R}}\left(2,2,2\right))=15$. The computations we have described above have been performed with the help of formal mathematical software.

We have shown that the quotiented Racah algebra $\overline{\mathcal{R}}\left(2,2,2\right)$ and the centralizer $\mathcal{C}_{111}$ have the same dimension, and moreover that the map $\overline{\phi}$ in Conjecture \ref{conj:t} is injective (alternatively, we already know that the map $\overline{\phi}$ is surjective). We conclude

\begin{thm}
Conjecture \ref{conj:t} is verified for $j_1=j_2=j_3=1$.
\end{thm}

\subsection{Connections with the Brauer algebra}
From the previous computations, one gets that the generator $C$ can be expressed in terms of $A$ and $B$ as follows
\begin{eqnarray}
 C&=&6-7A-B+A^2+\{A,B\}+\frac{1}{4}(ABA -A^2B-BA^2) \ .\label{eq:lem11} 
\end{eqnarray}
Therefore, we can suppress $C$ in the presentation of $\overline{\mathcal{R}}\left(2,2,2\right)$ to get a simpler presentation.
However, for the case treated in this section, we can do better by using the connection with the Brauer algebra. 
Indeed, it is known that this algebra is isomorphic to $\overline{\mathcal{R}}\left(2,2,2\right)$ \cite{LZ,LZ2}.
Let us recall the definition of the Brauer algebra:
\begin{definition} \cite{Brauer}
The Brauer algebra $B_3(\eta)$ is generated by $s_1$, $s_2$, $e_1$ and $e_2$ with the following defining relations
\begin{eqnarray}
&&s_1^2=1\quad,\quad s_2^2=1\quad,\quad e_1^2= \eta e_1\quad,\quad e_2^2=\eta e_2\quad,\quad s_1e_1=e_1s_1=e_1\quad,\quad s_2e_2=e_2s_2=e_2\ ,\\
&&s_1s_2s_1=s_2s_1s_2\quad,\quad e_1e_2e_1=e_1\quad,\quad e_2e_1e_2=e_2\ ,\\
&&s_1e_2e_1=s_2e_1\quad,\quad e_2e_1s_2=e_2s_1\ .
\end{eqnarray}
\end{definition}
The theorem below gives the precise connection between the Brauer algebra and the quotiented Racah algebra.
\begin{thm}\label{th:br}
 The quotiented Racah algebra  $\overline{\mathcal{R}}\left(2,2,2\right)$ is isomorphic to the Brauer algebra $B_3(3)$. This isomorphism is given explicitly by
\begin{eqnarray}
  \overline{\mathcal{R}}\left(2,2,2\right) & \rightarrow&B_3(3)\\
A &\mapsto& 2(s_1-e_1)+4\nonumber\\
B &\mapsto& 2(s_2-e_2)+4.\nonumber
\end{eqnarray}
\end{thm}
\proof 
Both algebras $\overline{\mathcal{R}}\left(2,2,2\right)$ and $B_3(3)$ are isomorphic to $C_{111}$. Therefore they are isomorphic. 
The explicit mapping is obtained because the images of $A$ (resp. $B$) and $2(s_1-e_1)+4$ (resp. $2(s_2-e_2)+4$) in $\text{End}([2]\otimes [2] \otimes[2])$ are equal.
\endproof

The inverse map is given by:
\[e_1\mapsto \frac{(A-2)(A-6)}{4}\,,\ \ \ \ \ s_1\mapsto \frac{A^2}{4}-\frac{3A}{2}+1\ ,\]
and similarly for $e_2,s_2$, in terms of $B$. We note that $e_1$ (respectively, $e_2)$ is 3 times the projector on the eigenspace of $A$ (respectively, of $B$) for the eigenvalue $0$ for the decomposition of $\overline{\mathcal{R}}\left(2,2,2\right)$ in a sum of eigenspaces of $A$ (respectively, $B$).

The image of the central element $C=6-7A-B+A^2+\{A,B\}+\frac{1}{4}(ABA -A^2B-BA^2)$ calculated explicitly in $B_3(3)$ simplifies (a lot) to:
\[C\mapsto 6+2(s_1-e_1)+2(s_2-e_2)+2s_1(s_2-e_2)s_1\ .\] 
This central element of $B_3(3)$ is equal to $12P_{+1}+6P_2+2P_3+0P_{-1}$ where $P_{\pm1}$ (respectively, $P_2$ and $P_3$) are the projectors associated to the irreducible representations of $B_3(3)$ of dimension 1, where $e_i\mapsto 0$ and $s_i\mapsto \pm1$ (respectively, of dimension 2 and of dimension 3).

As in the previous case, we can decompose the quotiented Racah algebra according to the eigenvalues of $C$.
Indeed, $\overline{\mathcal{R}}\left(2,2,2\right)$ is the direct sum of the algebras:
\begin{itemize}
 \item $C=12$, $A=6$ and $B=6$\ ;
 \item $C=6$, 
 \begin{eqnarray}
&& A(A-2)(A-6)=0\ ,\qquad B(B-2)(B-6)=0\ ,\\
&& (A+B-12)(A+B-10)(A+B-6)=0\ ,\\
&& 2BAB=AB^2+B^2A-2B^2-2\{A,B\}+24B\ ,\\
&& 2ABA=BA^2+A^2B-2A^2-2\{A,B\}+24A\ ;
\end{eqnarray}
 \item $C=2$, 
 \begin{eqnarray}
&& A(A-2)(A-6)=0\ ,\qquad B(B-2)(B-6)=0\ ,\\
&& (A+B-8)(A+B-6)(A+B-2)=0\ ,\\
&& 2BAB=AB^2+B^2A-2B^2-2\{A,B\}+16B\ ,\\
&& 2ABA=BA^2+A^2B-2A^2-2\{A,B\}+16A\ ;
\end{eqnarray}
  \item $C=0$, $A=2$ and $B=2$.
\end{itemize}
The first and the fourth algebras are of dimension $1$. After some algebraic manipulations not detailed here, one can show that the second is of dimension $4$ 
and the third is of dimension $9$. This is another way to find that 
$\text{dim}\left(\overline{\mathcal{R}}\left(2,2,2\right)\right)=15$.

\section{Quotiented Racah $\overline{\mathcal{R}}\left(j(j+1),\frac{3}{4},\frac{3}{4}\right)$ and classical one-boundary Temperley--Lieb algebra  \label{sec:sff}}

In this section, we focus on the case $j_1=j$ for $j=1,\frac{3}{2},2,\dots$ and $j_2=j_3=\frac{1}{2}$
and provide the description of the centralizer ${\mathcal{C}}_{j,\frac{1}{2},\frac{1}{2}}$ for $j=1,\frac{3}{2},2,\dots$ in terms of generators and relations by using the quotiented Racah algebra $\overline{\mathcal{R}}\left(j(j+1),\frac{3}{4},\frac{3}{4}\right)$. 
We prove that it is isomorphic to the one-boundary Temperley--Lieb algebra \cite{MS, MW,NRG}.

The Bratteli diagrams associated to the tensor product of two spin-$\frac{1}{2}$ and one spin-$j$ representations are displayed 
in Figures \ref{fig:jff} and \ref{fig:ffj}.
We read out that $\text{dim}({\mathcal{C}}_{j,\frac{1}{2},\frac{1}{2}})=1^2+2^2+1^2=6$.
The sets $\mathcal{J}$ are given by
\begin{equation}
 \mathcal{J}_{12}=\mathcal{J}_{13}= \left\{j-\frac{1}{2},j+\frac{1}{2}\right\}\quad\text{,} \qquad  \mathcal{J}_{23}=\{0,1\}
 \quad\text{,} \qquad \mathcal{J}_{123}=\left\{j-1,j,j+1 \right\}\ ,
\end{equation}
and the sets $\mathcal{M}$ by
\begin{eqnarray} 
 &&\mathcal{M}_{123}=\mathcal{M}_{132}=\left\{j+\frac{5}{4},\ -j-\frac{3}{4},\ j+\frac{1}{4},\ -j+\frac{1}{4}\right\}\ ,
\end{eqnarray}
and
\begin{eqnarray} 
 &&\mathcal{M}_{231}=\left\{j(j+3),\ (j+2)(j-1),\ j(j+1),\ (j+1)(j-2)\right\}\ .
\end{eqnarray}

\begin{figure}[htb]
\begin{center}
 \begin{tikzpicture}[scale=0.35]
\node at (0,0) {$[2j]$};

\draw (-1,-1.2) -- (-4,-3.5);\draw (1,-1.2) -- (4,-3.5);

\node at (-6,-4.5) {$[2j+1]$};\node at (6,-4.5) {$[2j-1]$};

\draw (-8,-5.2) -- (-12,-8.2);\draw (-4,-5.2) -- (-1,-8.2);\draw (4,-5.2) -- (1,-8.2);\draw (8,-5.2) -- (12,-8.2);

\node at (-14,-9) {$[2j+2]$};\node at (0,-9) {$[2j]$};\node at (14,-9) {$[2j-2]$};

\node at (-23,-2.5) {$\otimes$};\diag{-22}{-2}{1};\node at (-23,-6.5) {$\otimes$};\diag{-22}{-6}{1};

\end{tikzpicture}
\caption{Bratteli diagram $\mathcal{B}(j,\frac{1}{2},\frac{1}{2})$ ($j\geq 1$). \label{fig:jff}}
\end{center}
\end{figure}
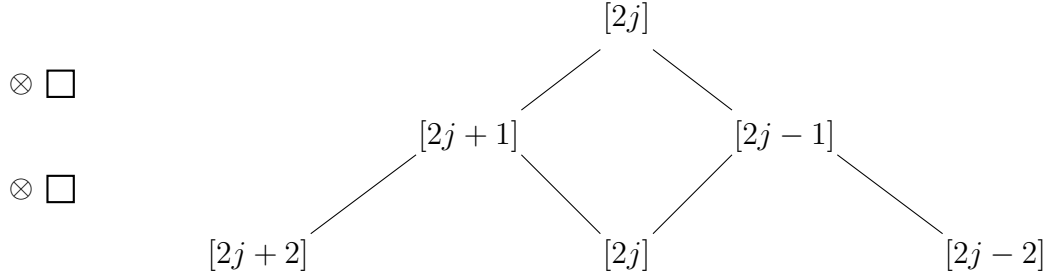

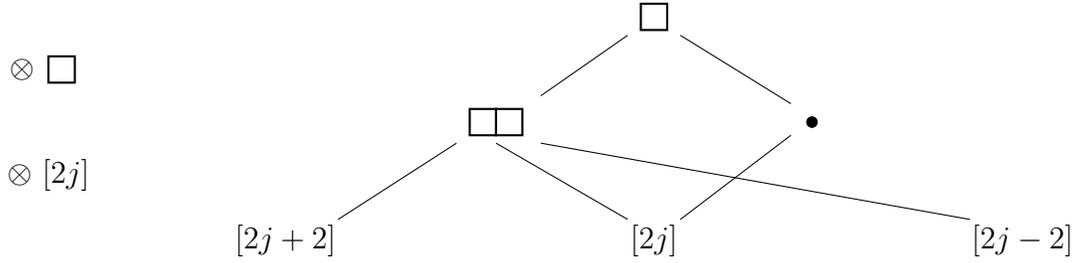
\begin{figure}[htb]
\begin{center}
 \begin{tikzpicture}[scale=0.35]
\diag{-0.5}{0}{1};

\draw (-1,-1.2) -- (-4.3,-3.5);\draw (1,-1.2) -- (5.2,-3.8);

\diag{-7}{-4}{2};\filldraw (6,-4.5) circle (2mm);  

\draw (-7.5,-5.3) -- (-12,-8.2);\draw (-6,-5.3) -- (-1,-8.2);\draw (5.2,-5) -- (1,-8.2);\draw (-4.3,-5.3) -- (12,-8.2);

\node at (-14,-9) {$[2j+2]$};\node at (0,-9) {$[2j]$};\node at (14,-9) {$[2j-2]$};

\node at (-24,-2.5) {$\otimes$};\diag{-23}{-2}{1};\node at (-23,-6.5) {$\otimes\ [2j]$};

\end{tikzpicture}
\caption{Bratteli diagram $\mathcal{B}(\frac{1}{2},\frac{1}{2},j)$  ($j\geq 1$). \label{fig:ffj}}
\end{center}
\end{figure}

By using the explicit contents given above of the sets $\mathcal{J}$ and $\mathcal{M}$ and by redefining the generators as follows
${\mathcal{A}}=\frac{1}{2j}((j+\frac{1}{2})(j+\frac{3}{2})-A)$, $ \mathcal{B}=2-B$ 
and $G=\frac{1}{2j}((j+1)(j+2)-C)$ with  $z=\frac{2j+1}{2j}$, the relations of the quotiented Racah $\overline{\mathcal{R}}\left(j(j+1),\frac{3}{4},\frac{3}{4}\right)$ become
\begin{eqnarray}
&& \mathcal{B}\mathcal{A} \mathcal{B}=2G+2\{\mathcal{A}, \mathcal{B}\}-4\mathcal{A}-G \mathcal{B}\ ,\label{eq:nc1}\\
&& 2\mathcal{A} \mathcal{B}\mathcal{A}=\left(3z-2\right)\big(\{\mathcal{A}, \mathcal{B}\}-z \mathcal{B})+2(2-z)\mathcal{A}+\left(3z-2-2\mathcal{A}\right)G\ , \label{eq:nc2}\\
&&  {\mathcal{A}}^2=z {\mathcal{A}}\ ,\qquad  \mathcal{B}^2=2 \mathcal{B} \ , \qquad G\left(G-2z\right)\left(G+1-2z\right)=0\ ,\label{eq:nc3}\\
&& ( G-\mathcal{A}+(1-z)\mathcal{B})\ ( G-\mathcal{A}+(1-z)\mathcal{B}-z)=0\label{eq:p1}\\
 && (\mathcal{A}+(z-1)\mathcal{B})\ (\mathcal{A}+(z-1)\mathcal{B}+1-2z)\ 
 (\mathcal{A}+(z-1)\mathcal{B}+1-z)\ (\mathcal{A}+(z-1)\mathcal{B}-z)=0 \ ,\label{eq:p2}\\
 && (G-\mathcal{A})\ (G-\mathcal{A}+1-2z)\ (G-\mathcal{A}+1-z)\ (G-\mathcal{A}-z)=0\ ,\label{eq:p3}\\
 && (G-(z-1)\mathcal{B})\ (G-(z-1)\mathcal{B}+1-2z)\ (G-(z-1)\mathcal{B}-1)\ (G-(z-1)\mathcal{B}-2z)=0 \ .\label{eq:p4}
\end{eqnarray}
We recall that $G$ is a central element.
To simplify the presentation, we need the following lemma.
\begin{lem}\label{lem:j11} The relations 
\begin{eqnarray}
 &&G=z\mathcal{B}+2\mathcal{A}-\{\mathcal{A},\mathcal{B}\}\ ,\label{lem:j1}\\
 &&  (G+1-2z)\mathcal{B}=0   \ ,   \label{lem:j2} \\ 
 && (G+1-2z)(G-2\mathcal{A})   =0\ .\label{lem:j3}
\end{eqnarray}
hold in $\overline{\mathcal{R}}\left(j(j+1),\frac{3}{4},\frac{3}{4}\right)$:
\end{lem}
\proof Observe that $\left(1+\frac{2}{z-2}\mathcal{A}  \right) \left(3z-2-2\mathcal{A}\right)=(3z-2)$ and multiply relation \eqref{eq:nc2} by the element $\left(1+\frac{2}{z-2}\mathcal{A}  \right)$ 
to find \eqref{lem:j1}. Note that $3z-2$ and $z-2$ do not vanish for $j=1,\frac{3}{2},\dots$.
Expanding \eqref{eq:p1}, one arrives at
\begin{equation}
 (G-z)(G-2\mathcal{A})+(z-1)\big(\{\mathcal{A} ,\mathcal{B}\}-2G\mathcal{B}+(3z-2)\mathcal{B}\big)   =0 \;.
\end{equation}
Replacing the anticommutator $\{\mathcal{A} ,\mathcal{B}\}$ and using \eqref{lem:j1} in the previous relation, one gets
\begin{equation}
 (G+1-2z)(G-2\mathcal{A}-2(z-1)\mathcal{B})  =0 \;.\label{eq:prt3}
\end{equation}
Expanding \eqref{eq:p3}, one finds
\begin{equation}
  (G+1-2z)\big( 2(z-1)G\mathcal{A}-2z(3z-1)\mathcal{A}+z(z+1)G\big)=0\;.
\end{equation}
Replacing $\mathcal{A}$ in the last factor by $G/2-(z-1)\mathcal{B}$ (because of \eqref{eq:prt3}) and then $G^2$ by $2zG$ (because of the third relation of \eqref{eq:nc3}),
one obtains \eqref{lem:j2}.
Using this result with \eqref{eq:prt3}, one proves relation \eqref{lem:j3}.
\endproof

In view of relation \eqref{lem:j1}, we can give a simpler presentation of $\overline{\mathcal{R}}(j(j+1),\frac{3}{4},\frac{3}{4})$ 
in which the generator $G$ is removed.
\begin{thm}\label{th:1BTL}
 The quotiented Racah algebra $\overline{\mathcal{R}}(j(j+1),\frac{3}{4},\frac{3}{4})$, for $j=1,\frac{3}{2},2,\dots$, is 
 generated by $\mathcal{A}$ and $\mathcal{B}$ subject to the following relations
\begin{eqnarray}
&& \mathcal{B} \mathcal{A} \mathcal{B}=\mathcal{B} \ ,\qquad
\mathcal{A}^2=z\mathcal{A}\ ,\qquad \mathcal{B}^2=2\mathcal{B}\ .\label{eq:i1}
\end{eqnarray}
\end{thm}
\proof We must first prove that the relations \eqref{eq:nc1}-\eqref{eq:p4} are equivalent to the relations \eqref{eq:i1}; this 
is straightforward 
with the help of Lemma \ref{lem:j11}.
Second, we must show that the relations \eqref{eq:i1} imply the defining relations \eqref{eq:nc1}-\eqref{eq:p4}. Let us 
suppose that $\mathcal{A}$ and $\mathcal{B}$ satisfy \eqref{eq:i1} and define 
\begin{eqnarray}
 G&=&z\mathcal{B}+2\mathcal{A}-\{\mathcal{A},\mathcal{B}\}\ .\label{eq:G1TL}
\end{eqnarray}
Multiplying this last relation on the left and on the right by $\mathcal{B}$, one gets (recalling that $\mathcal{B}^2=2\mathcal{B}$)
\begin{equation}
 G\mathcal{B}=2z\mathcal{B}-\mathcal{B}\mathcal{A}\mathcal{B}=2z\mathcal{B}-\mathcal{B}=\mathcal{B}G\ .\label{eq:GzB}
\end{equation}
This shows that $G$ commutes with $\mathcal{B}$ and, using  \eqref{eq:G1TL} to replace $z\mathcal{B}$, one recovers the relation \eqref{eq:nc1}.
Similarly, upon multiplying the relation \eqref{eq:G1TL} on the left and on the right by $\mathcal{A}$, one finds (using $\mathcal{A}^2=z\mathcal{A}$) that
\begin{equation}
 G\mathcal{A}=2z\mathcal{A}-\mathcal{A}\mathcal{B} \mathcal{A}=\mathcal{A}G\ .
\end{equation}
This proves that $G$ commutes with $\mathcal{A}$ and by adding $\frac{1}{2}(3z-2)$ times relation \eqref{eq:G1TL}, we obtain \eqref{eq:nc2}.
From the previous relations, one sees that
\begin{eqnarray}
 && G^2=(2z-1)(z\mathcal{B}-\{\mathcal{A},\mathcal{B}\})+2G\mathcal{A}=(2z-1)(G-2\mathcal{A})+2G\mathcal{A}
\end{eqnarray}
and one finds \eqref{lem:j3}. Combining with \eqref{eq:GzB}, we prove \eqref{eq:p1} and \eqref{eq:p3}.
This implies that
\begin{equation}
 (G+1-2z)(G-2z)G=(G+1-2z)(2\mathcal{A} -2z)\mathcal{A}=0\ ,
\end{equation}
which is the third relation of \eqref{eq:nc3}. The expansions of  \eqref{eq:p2} and \eqref{eq:p4} read as follows
\begin{eqnarray}
 &&\mathcal{A}  \mathcal{B}\mathcal{A}   \mathcal{B}  
+  \mathcal{B}\mathcal{A} \mathcal{B} \mathcal{A} 
+(z-2)\mathcal{B}\mathcal{A} \mathcal{B} -\{\mathcal{A},\mathcal{B}\} -(z-2)\mathcal{B} =0\ , \\
&&(G+1-2z)\big(  4(z-1)G -(4z-3)(2z-1)  \big) \mathcal{B}=0 \ .
\end{eqnarray}
Both relations are proven easily with the relations given above and this concludes the proof.
 \endproof

From this theorem, we easily obtain a generating set of the quotiented Racah algebra $\overline{\mathcal{R}}(j(j+1),\frac{3}{4},\frac{3}{4})$:
\begin{equation}
 1,\mathcal{A},\mathcal{B}, \mathcal{A}\mathcal{B},\mathcal{B} \mathcal{A}, \mathcal{A}\mathcal{B} \mathcal{A}\;.
\end{equation}
The dimension of $\overline{\mathcal{R}}(j(j+1),\frac{3}{4},\frac{3}{4})$ is therefore less or equal to 6. Since the dimension of the centralizer ${\mathcal{C}}_{j,\frac{1}{2},\frac{1}{2}}$ is 6 (for $j\geq 1$), and since we already know that the map $\overline{\phi}$ of Conjecture \ref{conj:t} is surjective by Proposition \ref{cor:sur}, we can immediately conclude that the set given above is a basis and that the dimension of $\overline{\mathcal{R}}(j(j+1),\frac{3}{4},\frac{3}{4})$ is 6. We thus have the following result:

\begin{thm}
Conjecture \ref{conj:t} is verified for $j_2=j_3=\frac{1}{2}$ and any $j_1\in\frac{\mathbb{N}}{2}$.
\end{thm}

\subsection{Connections with the one-boundary Temperley--Lieb algebra}
 
The algebra known as the one-boundary Temperley--Lieb algebra is generated by $\overline{\sigma}_0$ and $\overline{\sigma}_1$ with the defining relations \cite{MS, MW,NRG}
 \begin{equation}
 \overline{\sigma}_0^2=\frac{\sin(\omega)}{\sin(\omega+\gamma)} \overline{\sigma}_0\ , \quad  \overline{\sigma}_1^2=2\cos(\gamma)\overline{\sigma}_1\ , 
 \quad  \overline{\sigma}_1 \overline{\sigma}_0\overline{\sigma}_1= \overline{\sigma}_1\;.  \label{eq:1btl}
\end{equation} 
By setting $\omega=\hbar(2j+1)$ and $\gamma=-\hbar$ and by performing the limit $\hbar\rightarrow 0$, called the classical limit, in these defining relations, we arrive at the following definition:
\begin{definition}
 The classical one-boundary Temperley--Lieb algebra $\text{btl}(j)$ is generated by $\sigma_0$ and $\sigma_1$ subject to
\begin{equation}
 \sigma_0^2=\frac{2j+1}{2j} \sigma_0\ , \quad  \sigma_1^2=2\sigma_1\ , \quad  \sigma_1 \sigma_0\sigma_1= \sigma_1\;.
\end{equation}
\end{definition} We can now formulate the following theorem:
\begin{thm}
 The quotiented Racah algebra  $\overline{\mathcal{R}}\left( j(j+1),\frac{3}{4},\frac{3}{4}\right)$ is isomorphic to the classical one-boundary Temperley--Lieb algebra $\text{btl}(j)$. 
 This isomorphism is given explicitly by
\begin{eqnarray}
  \overline{\mathcal{R}}\left(j(j+1),\frac{3}{4},\frac{3}{4}\right) & \rightarrow& \text{btl}(j)\\
\cA &\mapsto& \sigma_0\nonumber\\
\cB &\mapsto& \sigma_1\nonumber
\end{eqnarray}
\end{thm}

The image in $\text{btl}(j)$ of the central element $G=\frac{1}{2j}((j+1)(j+2)-C)=z\mathcal{B}+2\mathcal{A}-\{\mathcal{A},\mathcal{B}\}$ is explicitly determined to be:
\[G\mapsto \frac{2j+1}{2j}\sigma_1+2\sigma_0-\sigma_0\sigma_1-\sigma_1\sigma_0\ .\] 
The algebra $\text{btl}(j)$ has 2 representations of dimension 1, given by $\sigma_1\mapsto 0$ and $\sigma_0\mapsto0$ or $\frac{2j+1}{2j}$. The above central element of $\text{btl}(j)$ is equal to $\frac{2j+1}{j}P_{1}+\frac{j+1}{j}P_{2}+0P_{1'}$ where $P_{1}$ (respectively, $P_{1'}$) is the central projector associated to the $\text{btl}(j)$ representation of dimension 1 with $\sigma_0\mapsto \frac{2j+1}{2j}$ (respectively, $\sigma_0\mapsto 0$) and $P_2$ is the central projector associated to the $\text{btl}(j)$ representation of dimension 2.

\section{Quotiented Racah $\overline{\mathcal{R}}\left(\frac{3}{4},2,2\right)$ and $\mathcal{C}_{\frac{1}{2},1,1}$ \label{sec:f11}} 
 
 The case $j_1=\frac{1}{2}$ and $j_2=j_3=1$, i.e. the centralizer $\mathcal{C}_{\frac{1}{2},1,1}$,  is the object of this section.
 From the Bratteli diagrams displayed in Figures \ref{fig:f11a} and \ref{fig:f11b}, we determine that $\text{dim}(\mathcal{C}_{\frac{1}{2},1,1})=1+2^2+2^2=9$, 
 $\mathcal{J}_{12}=\mathcal{J}_{13}=\{\frac{3}{4},\frac{15}{4}\}$, 
 $\mathcal{J}_{23}=\{0,2,6\}$, $\mathcal{J}_{123}=\{\frac{3}{4},\frac{15}{4},\frac{35}{4}\}$, $\mathcal{M}_{123}=\mathcal{M}_{132}=\{-3,0,3,5\}$ 
 and $\mathcal{M}_{231}=\{-\frac{9}{4},-\frac{5}{4},\frac{3}{4},\frac{7}{4},\frac{11}{4}\}$.
 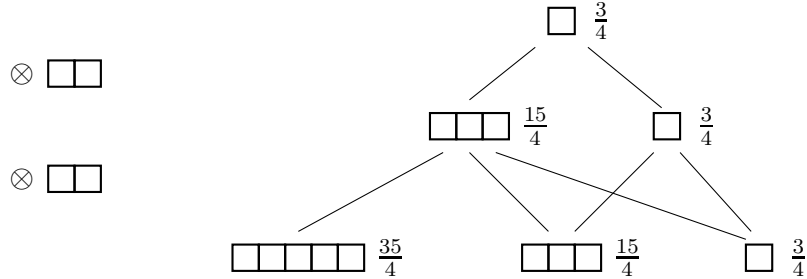
\begin{figure}[htb]
\begin{center}
 \begin{tikzpicture}[scale=0.35]
\diag{0}{0}{1};\node at (2,-0.5) {$\frac{3}{4}$};

\draw (-0.5,-1.5) -- (-3,-3.5);\draw (1.5,-1.5) -- (4.3,-3.8);

\diag{-4.5}{-4}{3};\node at (-0.5,-4.5) {$\frac{15}{4}$};\diag{4}{-4}{1};  \node at (6,-4.5) {$\frac{3}{4}$};

\draw (-4,-5.5) -- (-9.5,-8.5);\draw (-3,-5.5) -- (0,-8.5);\draw (4,-5.5) -- (1,-8.5); \draw (-2,-5.5) -- (7.5,-8.8);\draw (5,-5.5) -- (7.7,-8.5);

\diag{-12}{-9}{5};\node at (-6,-9.5) {$\frac{35}{4}$};
\diag{-1}{-9}{3};\node at (3,-9.5) {$\frac{15}{4}$};\diag{7.5}{-9}{1};\node at (9.5,-9.5) {$\frac{3}{4}$};

\node at (-20,-2.5) {$\otimes$};\diag{-19}{-2}{2};\node at (-20,-6.5) {$\otimes$};\diag{-19}{-6}{2};

\end{tikzpicture}
\caption{Bratteli diagram $\mathcal{B}(\frac{1}{2},1,1)$. On the right of each Young tableau, 
the corresponding value of the Casimir is recalled.\label{fig:f11a}}
\end{center}
\end{figure}

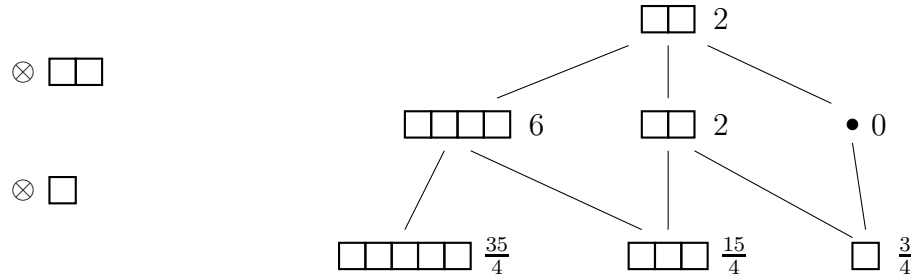
\begin{figure}[htb]
\begin{center}
 \begin{tikzpicture}[scale=0.35]
\diag{-0.5}{0}{2};\node at (2.5,-0.5) {$2$};
\draw (-1,-1.5) -- (-6,-3.5);\draw (0.5,-1.5)--(0.5,-3.5);\draw (2,-1.5) -- (6.7,-3.7);

\diag{-9.5}{-4}{4};\node at (-4.5,-4.5) {$6$};
\diag{-0.5}{-4}{2};\node at (2.5,-4.5) {$2$};
\filldraw (7.5,-4.5) circle (2mm); ; \node at (8.5,-4.5) {$0$};

\draw (-8,-5.5)-- (-9.5,-8.5);
\draw (-7,-5.5)-- (-0.5,-8.5);
\draw (0.5,-5.5) -- (0.5,-8.5);\draw (1.5,-5.5) -- (7.5,-8.8); \draw (7.5,-5.2) -- (8,-8.5);

\diag{-12}{-9}{5};\node at (-6,-9.5) {$\frac{35}{4}$};
\diag{-1}{-9}{3};\node at (3,-9.5) {$\frac{15}{4}$};\diag{7.5}{-9}{1};\node at (9.5,-9.5) {$\frac{3}{4}$};

\node at (-24,-2.5) {$\otimes$};\diag{-23}{-2}{2};
\node at (-24,-7) {$\otimes$};\diag{-23}{-6.5}{1};
\end{tikzpicture}
\caption{Bratteli diagram $\cB(1,1,\frac{1}{2})$. 
The value of the Casimir is recalled on the right of each corresponding Young tableau.\label{fig:f11b}  }
\end{center}
\end{figure}

Given these explicit sets $\mathcal{J}$ and $\mathcal{M}$ and
in terms of the redefined generators ${\mathcal{A}}=A+\frac{1}{4}$ and $G=C+\frac{1}{4}$, 
the defining relations of the quotiented Racah $\overline{\mathcal{R}}\left(\frac{3}{4},2,2\right)$ become
\begin{eqnarray}
&& 2B\mathcal{A} B= \mathcal{A}B^2+B^2\mathcal{A}-2B^2-2\{\mathcal{A},B\}+10B+2GB  \ ,\label{eq:ne1}\\
&& 2\mathcal{A} B\mathcal{A}=3\{\mathcal{A},B\}-7B+2G\mathcal{A}+2G\ , \label{eq:ne2}\\
&&  (\mathcal{A}-1) (\mathcal{A}-4)=0\ ,\qquad  B(B-2)(B-6)=0 \ , \qquad \left(G-1\right)\left(G-4\right)\left(G-9\right)=0\ ,\label{eq:ne3}\\
&& ( G-\mathcal{A}-B+4)\ ( G-\mathcal{A}-B+1)=0\label{eq:pp1}\\
 && (\mathcal{A}+B-10)\  (\mathcal{A}+B-8)\  (\mathcal{A}+B-5)\  (\mathcal{A}+B-2)=0 \ ,\label{eq:pp2}\\
 && (G-\mathcal{A}-5)\ (G-\mathcal{A}-3)\ (G-\mathcal{A})\ (G-\mathcal{A}+3)=0\ ,\label{eq:pp3}\\
 && (G-B-3)\ (G-B-2)\ (G-B-1)\ (G-B+1)\ (G-B+2)\ =0 \ .\label{eq:pp4}
\end{eqnarray}
We recall that $G$ is a central element.

We want to demonstrate that $\text{dim}(\overline{\mathcal{R}}\left(\frac{3}{4},2,2\right))=9$ and shall use to that end the same approach as in Section \ref{sec:brauer}. Let
\begin{equation}
 \mathcal{S}=\{ 1, \mathcal{A},B,B^2,\mathcal{A}B,B\mathcal{A},\mathcal{A}B^2,\mathcal{A}B\mathcal{A},B\mathcal{A}B\}\ .
\end{equation}
We can show after some algebraic manipulations that 
\begin{equation}
 \mathcal{S}_r=\mathcal{S} \cup G\mathcal{S} \cup G^2\mathcal{S}
\end{equation}
is a generating set. We can therefore construct the 27 by 27 matrices $\mathcal{A}_r$, $B_r$ and $G_r$ 
corresponding to the regular actions of $\mathcal{A}$, $B$ and $G$ on the set $\mathcal{S}_r$.
By asking that $\mathcal{A}_r$, $B_r$ and $G_r$ satisfy \eqref{eq:ne1}--\eqref{eq:pp4}, we find $18$ independent constraints between the elements of $\mathcal{S}_r$ and can thus reduce the generating set from $\mathcal{S}_r$ to $\mathcal{S}$. 

Moreover, we can show that the images in $\text{End}([1]\otimes [2] \otimes [2])$ under the homomorphism $\overline \phi$ (see Conjecture \ref{conj:t}) of the 9 elements of $\mathcal{S}$ are 9 linearly independent matrices.
This proves that $\text{dim}(\overline{\mathcal{R}}\left(\frac{3}{4},2,2\right))=9$. The computations have been done by using formal mathematical software.

We thus observe that the quotiented Racah algebra $\overline{\mathcal{R}}\left(\frac{3}{4},2,2\right)$ and the centralizer $\mathcal{C}_{122}$ have the same dimension, and moreover that the map $\overline{\phi}$ in Conjecture \ref{conj:t} is injective (we already know from Proposition \ref{cor:sur} that $\overline{\phi}$ is surjective). We conclude that

\begin{thm}
Conjecture \ref{conj:t} is verified for $j_1=\frac{1}{2}$ and $j_2=j_3=1$.
\end{thm}

\subsection{A one-boundary Brauer algebra}

The computations show that the generator $G$ can be expressed as follows in terms of $\mathcal{A}$ and $B$ 
\begin{equation}\label{eq:lem12}
G= \frac{3}{2}B-\frac{1}{2}\{\mathcal{A},B\}+\frac{1}{4}\mathcal{A}B\mathcal{A}\ .
\end{equation}
We have also
\begin{equation}\label{eq:lem13} 
 GB= -8+2\mathcal{A}+8B-\frac{1}{2}B^2-2\{\mathcal{A},B\}+\frac{1}{2}\mathcal{A}B\mathcal{A} +\frac{1}{2}B\mathcal{A}B\ .
\end{equation}
We can therefore eliminate $G$ from the presentation of $\overline{\mathcal{R}}\left(\frac{3}{4},2,2\right)$.
\begin{prop}
 The quotiented Racah algebra $\overline{\mathcal{R}}\left(\frac{3}{4},2,2\right)$ is generated by $\mathcal{A}$ and $B$ subject to 
 \begin{eqnarray}
 &&  (\mathcal{A}-1) (\mathcal{A}-4)=0\ ,\qquad  B(B-2)(B-6)=0 \ ,\label{eq:n0} \\
 && B\mathcal{A}B-\mathcal{A}B\mathcal{A}=\mathcal{A}B^2+B^2\mathcal{A}-3B^2-6\{\mathcal{A},B\}+26B-16+4\mathcal{A}\ ,\label{eq:n1}\\
 &&  B\mathcal{A}B^2+16\mathcal{A}B-2\mathcal{A}B^2-8B\mathcal{A}B+12B\mathcal{A}+6B^2-48B-24\mathcal{A}+ 72=0 \ .\label{eq:n2}
 \end{eqnarray}
\end{prop}
\proof Relation \eqref{eq:n1} is implied by \eqref{eq:ne1} and \eqref{eq:lem13} whereas relation \eqref{eq:n2} is obtained from \eqref{eq:pp2}.
We can show that the dimension of the algebra generated by $\mathcal{A}$ and $B$ subject to \eqref{eq:n0}-\eqref{eq:n2} is $9$
which concludes the proof.
\endproof

Let us now introduce a new algebra.
\begin{definition}
 The one-boundary Brauer algebra $bB$ is generated by $e_0$, $e_1$ and $s_1$ subject to
 \begin{eqnarray}
  && e_0^2=\frac{3}{2}e_0\ , \quad   s_1^2=1\ , \quad e_1^2=3e_1 \ , \quad s_1e_1=e_1s_1=e_1 \ , \\
  && e_1e_0s_1  =e_1-e_1e_0\ , \quad s_1e_0e_1=e_1-e_0e_1\ ,\\
  && 4e_0s_1e_0=1+2e_0+e_1-s_1 +2\{e_0,s_1-e_1\} -2s_1e_0s_1\ , \\
   && 4e_0e_1e_0=-1+2e_0+s_1-e_1 -2\{e_0,s_1-e_1\} +2s_1e_0s_1\ .
 \end{eqnarray}
\end{definition}
The name of this algebra has been chosen  in view of the similarity with the one-boundary Temperley--Lieb algebra since the relation
\begin{equation}
 e_1e_0e_1=\frac{3}{2}e_1 \ 
\end{equation}
holds in the one-boundary Brauer algebra.
Because of that one can show that the one-boundary Brauer algebra is of dimension $9$ with 
a basis given by 
\begin{equation}
 1, e_0,e_1,s_1,e_0e_1,e_1e_0,e_0s_1,s_1e_0,s_1e_0s_1\ .
\end{equation}

\begin{thm}
 The map 
 \begin{eqnarray}
  \overline{\mathcal{R}}\left(\frac{3}{4},2,2\right) &\to& bB \label{eq:mapf}\\
  \mathcal{A} &\mapsto& 4-2e_0\nonumber\\
  B  &\mapsto& 2(s_1-e_1)+4\nonumber
 \end{eqnarray}
is an algebra isomorphism.
\end{thm}
\proof The homomorphism is proved by direct computation. The map is surjective since the image of $\frac{1}{4}B$ is $4s_1-3e_1+5$. The bijection is obtained since the dimension of both algebras are the same.    \endproof

The inverse map is given by:
\[e_0\mapsto 2-\frac{\cA}{2}\,,\ \ \ \ e_1\mapsto \frac{(B-2)(B-6)}{4}\,,\ \ \ \ \ s_1\mapsto \frac{B^2}{4}-\frac{3B}{2}+1\ .\]
We note as in Section \ref{sec:brauer} that $e_1$ is identified with 3 times the projector on the eigenspace of $B$ with eigenvalue $0$ relative to the decomposition of $\overline{\mathcal{R}}\left(\frac{3}{4},2,2\right)$ in a sum of $B$ - eigenspaces. Similarly, $e_0$ is identified with $\frac{3}{2}$ times the projector on the $A$ - eigenspace with eigenvalue $1$.

The image in $bB$ of the central element $G=C+\frac{1}{4}=\frac{3}{2}B-\frac{1}{2}\{\mathcal{A},B\}+\frac{1}{4}\mathcal{A}B\mathcal{A}$ simplifies to:
\[G\mapsto 7-2e_0+2(s_1-e_1)-2s_1e_0s_1\ \] 
when calculated explicitly.
The algebra $bB$ has a single one-dimensional representation $(e_0,e_1\mapsto 0$ and $s_1\mapsto 1$) and 2 irreducible representations of dimension 2, say $V$ and $V'$. They are distinguished by the value of the central element given above (the image of $G$): in one, say $V$, it is equal to 4 and in the other, $V'$,  it is equal to 1. The central element is equal to $9P_1+4P_V+P_{V'}$, where $P_{1}$, $P_V$ and $P_{V'}$ are the central projectors associated to the irreducible representations of $bB$ (with obvious notations).\\

Let us remark that the image in $\text{End}([1] \otimes [2] \otimes [2])$ of $B$ and of $\mathcal{A}$ under $\overline \phi$ acts non-trivially on $[2] \otimes [2]$ and 
 $[1] \otimes [2]$ respectively. It is natural to define the algebra $bB$ the way we did since the map \eqref{eq:mapf} then gives an image of $B$ (resp. $\mathcal{A}$) which is the same as in Section \ref{sec:brauer} (resp. Section \ref{sec:sff}).

\section{Quotiented Racah $\overline{\mathcal{R}}\left(\frac{15}{4},\frac{15}{4},\frac{15}{4}\right)$ 
and ${\mathcal{C}}_{\frac{3}{2},\frac{3}{2},\frac{3}{2}}$\label{sec:32}}

We consider in this section the case $j_1=j_2=j_3=\frac{3}{2}$ and show that the centralizer ${\mathcal{C}}_{\frac{3}{2},\frac{3}{2},\frac{3}{2}}$ is isomorphic to
 $\overline{\mathcal{R}}\left(\frac{15}{4},\frac{15}{4},\frac{15}{4}\right)$.
 
The Bratteli diagram associated to the tensor product of three spin-$\frac{3}{2}$ representations is displayed in Figure \ref{fig:32}.
It reveals that $\text{dim}({\mathcal{C}}_{\frac{3}{2},\frac{3}{2},\frac{3}{2}})=1^2+2^2+3^2+4^2+2^2= 34$, $\mathcal{J}_{12}=\mathcal{J}_{13}=\mathcal{J}_{23}=\{0,2,6,12\}$, 
$\mathcal{J}_{123}=\{\frac{3}{4},\frac{15}{4},\frac{35}{4},\frac{63}{4},\frac{99}{4}\}$ and 
$\mathcal{M}_{123}=\mathcal{M}_{231}=\mathcal{M}_{132}=\{-\frac{33}{4},-\frac{21}{4},-\frac{13}{4},-\frac{9}{4},-\frac{5}{4},\frac{7}{4},\frac{11}{4},\frac{15}{4},\frac{27}{4},\frac{39}{4},\frac{51}{4}\}$.
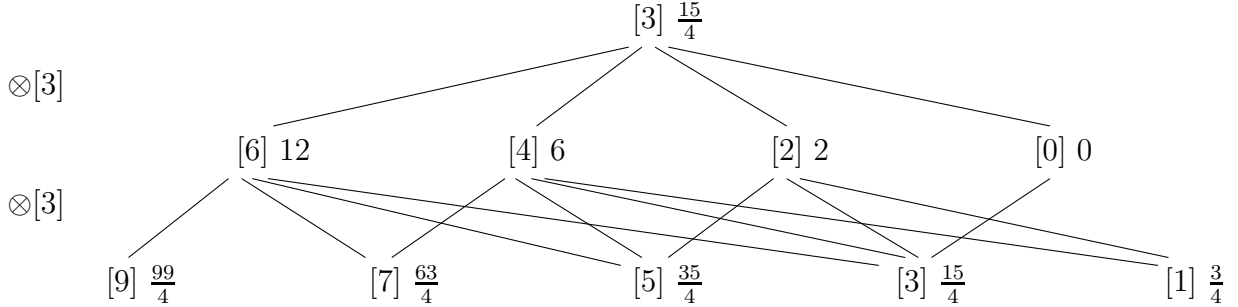
\begin{figure}[htb]
\begin{center}
 \begin{tikzpicture}[scale=0.35]
\node at (0,0) {$[3]\ \frac{15}{4}$};

\draw (-1.5,-1) -- (-15,-4);\draw (-1,-1) -- (-5,-4);\draw (-0.5,-1) -- (4.5,-4);\draw (0,-1) -- (14.5,-4);

\node at (-15,-5) {$[6]\ 12$};
\node at (-5,-5) {$[4]\ 6$};\node at (5,-5) {$[2]\ 2$};\node at (15,-5) {$[0]\ 0$};

\draw (-16.7,-6) -- (-20.5,-9); \draw (-16.2,-6) -- (-11.5,-9); \draw (-15.8,-6) -- (-1.8,-9.3); \draw (-15.2,-6) -- (8,-9.3); 

\draw (-6.2,-6) -- (-10.5,-9); \draw (-5.8,-6) -- (-1,-9); \draw (-5.2,-6) -- (9,-9);  \draw (-4.7,-6) -- (18.6,-9.3); 

 \draw (4,-6) -- (0,-9); \draw (4.5,-6) -- (9.5,-9);  \draw (5,-6) -- (19,-9);
 
 \draw (14.5,-6) -- (10,-9);

\node at (-20,-10) {$[9]\ \frac{99}{4}$};
\node at (-10,-10) {$[7]\  \frac{63}{4}$};\node at (0,-10) {$[5]\  \frac{35}{4}$};\node at (10,-10) {$[3]\  \frac{15}{4}$};\node at (20,-10) {$[1]\  \frac{3}{4}$};

\node at (-24,-2.5) {$\otimes [3]$};
\node at (-24,-7) {$\otimes [3]$};
\end{tikzpicture}
\caption{Bratteli diagram $\cB(\frac{3}{2},\frac{3}{2},\frac{3}{2})$. 
The value of the Casimir is recalled on the right of each corresponding Young tableau.\label{fig:32}  }
\end{center}
\end{figure} 
If we set $G=C+\frac{1}{4}$, the characteristic polynomial of $G$ is given by 
\begin{equation}
 (G-1)(G-4)(G-9)(G-16)(G-25)=0\ .
\end{equation}
We shall compute the dimensions of the different quotients of $\overline{\mathcal{R}}\left(\frac{15}{4},\frac{15}{4},\frac{15}{4}\right)$ by the relation $G=1$, $G=4$, $G=9$, $G=16$ or $G=25$.
When the value of the central element $G$ (or $C$) is equal to a parameter $g$, 
the relations $\prod_{j\in \mathcal{J}_{12}} ( A -j(j+1))=0$ and $\prod_{m\in \mathcal{M}_{123}}\big( g-\frac{1}{4} - A -m\big)=0$ of the quotiented 
Racah algebra reduce to only one relation  $\prod_{m\in \overline{\mathcal{M}}_{123}}( A -m)=0$ where 
\begin{equation}
 \overline{\mathcal{M}}_{123}= \{j(j+1) \ |\ j\in \mathcal{J}_{12} \} \cap \{m -g+\frac{1}{4} \ |\ m\in \mathcal{M}_{123} \}\ .
\end{equation}
Similarly, the second relation in \eqref{eq:quo1} and the second ones in \eqref{eq:quo5} reduce to only one characteristic polynomial for $B$
and relations \eqref{eq:quo2} and \eqref{eq:quo3} reduce to one characteristic polynomial for $A+B$.
Therefore, for the different cases, one gets:
\begin{itemize}
 \item $G=25$ 
  \begin{eqnarray}
&&A=B=12
 \end{eqnarray}

 \item $G=16$
 \begin{eqnarray}
 && (A-6)(A-12)=0\quad ,\qquad(B-6)(B-12)=0\quad ,\qquad (A+B-21)(A+B-15)=0\ ,\\
 && BAB=8\{A,B\}+72A+9B+72\quad ,\qquad ABA=8\{A,B\}+72B+9A+72
 \end{eqnarray}
 \item $G=9$
 \begin{eqnarray}
 && (A-2)(A-6)(A-12)=0\quad ,\qquad(B-2)(B-6)(B-12)=0\\
 && (A+B-18)(A+B-14)(A+B-8)=0\ ,\\
 && 2BAB-B^2A-AB^2=-2B^2-2\{A,B\}+40B\ ,\\ 
 && 2ABA-A^2B-BA^2=-2A^2-2\{A,B\}+40A
 \end{eqnarray}
 \item $G=4$
 \begin{eqnarray}
  &&A(A-2)(A-6)(A-12)=0\quad ,\qquad B(B-2)(B-6)(B-12)=0\ ,\\
  &&(A+B-15)(A+B-13)(A+B-9)(A+B-3)=0\ , \\
 && 2BAB-B^2A-AB^2=-2B^2-2\{A,B\}+30B\ ,\\ 
 && 2ABA-A^2B-BA^2=-2A^2-2\{A,B\}+30A
 \end{eqnarray}
 \item $G=1$
 \begin{eqnarray}
 && (A-2)(A-6)=0\quad ,\qquad(B-2)(B-6)=0\quad ,\qquad (A+B-6)(A+B-10)=0\\
 && BAB=12A+28B-96\quad,\quad ABA=12B+28A-96
 \end{eqnarray}
\end{itemize}
We can show that for both cases $G=16$ and $G=1$, the dimension of the algebra is 4 and a basis is $\{1,A,B,AB\}$.
For the cases $G=4$ and $G=9$, we used formal mathematical software to prove that the dimension is $16$ and $9$, respectively.

We conclude that the sum of these dimensions is 34 which proves that $\overline{\mathcal{R}}\left(\frac{15}{4},\frac{15}{4},\frac{15}{4}\right)$ has the same dimension as the centralizer $\mathcal{C}_{\frac{3}{2},\frac{3}{2},\frac{3}{2}}$. This concludes the proof of the conjecture \ref{conj:t} in this case since we already know that the map $\overline{\phi}$ is surjective by Proposition \ref{cor:sur}.
\begin{thm}
Conjecture \ref{conj:t} is verified for $j_1=j_2=j_3=\frac{3}{2}$.
\end{thm}

\section{Quotiented Racah $\overline{\mathcal{R}}\left(j(j+1),\frac{3}{4},k(k+1)\right)$ 
and ${\mathcal{C}}_{j,\frac{1}{2},k}$\label{sec:j12k}}

We consider the situation $\{j_1,j_2,j_3\}=\{\frac{1}{2},j,k\}$ where $j$ and $k$ are any positive integers or half-integers. We assume that $(j,k)\neq(\frac{1}{2},\frac{1}{2})$ since this case was treated in full details earlier. From the $S_3$-invariance of the conjecture explained in Section \ref{sec:S3}, there is no loss in generality to take $\alpha_1=j(j+1)$, $\alpha_2=\frac{3}{4}$, $\alpha_3=k(k+1)$ and to assume that $j\geq k$. 

For simplicity of notations, we define the following new parameters:
\[x:=j+\frac{1}{2}\ \ \ \ \text{and}\ \ \ \ \ y:=k+\frac{1}{2}\ .\]
The characteristic equations for $A$ and $B$ in Conjecture \ref{conj:t} then become:
\begin{equation}\label{quadr1}
A^2=2x^2A-x^2(x^2-1)\ \ \ \ \ \ \ \text{and}\ \ \ \ \ \ \ B^2=2y^2B-y^2(y^2-1)\ .
\end{equation}
Let $c$ be a complex number. We denote by $H_{j,k,c}$ the quotient of the universal Racah algebra $\mathcal{R}(j(j+1),\frac{3}{4},k(k+1))$ by the relations \eqref{quadr1} together with $C=c$. In words, these relations replace the central element $C$ by the number $c$, and force $A$ and $B$ 
to be canceled by a polynomial of order 2 with the simple eigenvalues:
$$Sp(A)=\{(j-\frac{1}{2})(j+\frac{1}{2}),(j+\frac{1}{2})(j+\frac{3}{2})\}\ \ \ \ \text{and}\ \ \ \ Sp(B)=\{(k-\frac{1}{2})(k+\frac{1}{2}),(k+\frac{1}{2})(k+\frac{3}{2})\}\ .$$

The main step for proving the conjecture in this situation is the following proposition.
\begin{prop}\label{prop:dim4}
The algebra $H_{j,k,c}$ is of dimension smaller or equal to 4.
\end{prop}
\begin{proof}
Using the characteristic equations \eqref{quadr1} for $A$ and $B$, we rewrite the defining relations \eqref{eq:UR1} and \eqref{eq:UR2} and we obtain the following relations in $H_{j,k,c}$:
\begin{eqnarray}
&&ABA=(x^2-1)\{A,B\}-x^2(x^2-1)B+(y^2-x^2+c+\frac{1}{4})A+(x^2-1)(x^2+y^2-c-\frac{1}{4})\,,\label{eq:UR1mod}\\
&&BAB=(y^2-1)\{A,B\}-y^2(y^2-1)A+(x^2-y^2+c+\frac{1}{4})B+(y^2-1)(x^2+y^2-c-\frac{1}{4})\,. \label{eq:UR2mod}
 \end{eqnarray}
The first of these relations show that
\[(A-(x^2-1))BA\in\text{Span}\{1,A,B,AB\}\ .\]
The roots of the characteristic polynomials of $A$ are $x^2\pm x$. Moreover, since $(j,k)\neq(\frac{1}{2},\frac{1}{2})$ and $j\geq k$, we have that $x\neq 1$ and therefore that $A-(x^2-1)$ is invertible in $\mathbb{C}[A]$. Moreover, it is clear that the subspace $\text{Span}\{1,A,B,AB\}$ is stable under multiplication on the left by $A$, so we conclude that
\[BA\in\text{Span}\{1,A,B,AB\}\ .\]
Now it is immediate from this fact and the defining relations that $\text{Span}\{1,A,B,AB\}$ is stable under multiplication on the left by the generators $A$ and $B$. This shows that this subspace is a subalgebra of $H_{j,k,c}$, and as it contains moreover the two generators $A$ and $B$, it must be the whole algebra $H_{j,k,c}$. We conclude that $H_{j,k,c}$ has a spanning set with 4 elements thereby validating the proposition.
\end{proof}

We are now ready to prove the conjecture in this situation.
\begin{thm}
Conjecture \ref{conj:t} is true for $\{j_1,j_2,j_3\}=\{\frac{1}{2},j,k\}$.
\end{thm}
\begin{proof}
As explained in the beginning of this section, we can assume that $j_1=j\geq k=j_3$, $j_2=\frac{1}{2}$ and moreover $j>\frac{1}{2}$. The Bratteli diagram in this situation gives that here
\[\mathcal{J}_{123}\subseteq\{j-k-\frac{1}{2},j-k+\frac{1}{2},\dots,\dots,j+k-\frac{1}{2},j+k+\frac{1}{2}\}\ ,\]
this being an equality if $j>k$, while the negative value $j-k-\frac{1}{2}$ must be removed if $j=k$. Moreover, the multiplicities of the representations in the third line are equal to 2 except for $j-k-\frac{1}{2}$ (if present) or $j+k+\frac{1}{2}$ for which it is 1.

From the characteristic equations for $C$ imposed in $\overline{\mathcal{R}}\left(j(j+1),\frac{3}{4},k(k+1)\right)$, we have:
\[\overline{\mathcal{R}}\bigl(j(j+1),\frac{3}{4},k(k+1)\bigr)=\bigoplus_{x\in\mathcal{J}_{123}}\overline{\mathcal{R}}\bigl(j(j+1),\frac{3}{4},k(k+1)\bigr)_x\ ,\]
where $\overline{\mathcal{R}}\left(j(j+1),\frac{3}{4},k(k+1)\right)_x$ denotes the quotiented Racah algebra where $C$ is replaced by $x(x+1)$ (the quotient by $C=x(x+1)$).

We are going to show that, for $x\in\mathcal{J}_{123}$, we have
\[\dim\Bigl(\overline{\mathcal{R}}\bigl(j(j+1),\frac{3}{4},k(k+1)\bigr)_x\Bigr)\leq\left\{\begin{array}{cc}
4 & \text{if $x\notin\{j-k-\frac{1}{2},j+k+\frac{1}{2}\}$,}\\
1 & \text{otherwise.}
\end{array}\right.\]
This implies at once that the dimension of $\overline{\mathcal{R}}\bigl(j(j+1),\frac{3}{4},k(k+1)\bigr)$ is smaller or equal to the dimension of the 
centralizer ${\mathcal{C}}_{j,\frac{1}{2},k}$ and so, from the surjectivity of the map proved in Proposition \ref{cor:sur}, this is enough to prove the conjecture. 

\medskip
The algebra $\overline{\mathcal{R}}\bigl(j(j+1),\frac{3}{4},k(k+1)\bigr)_x$ is a quotient of the algebra $H_{j,k,x(x+1)}$ studied earlier (the quotient by the 
remaining relations \eqref{eq:quo2}--\eqref{eq:quo5} in Conjecture \ref{conj:t}). So we have immediately from Proposition \ref{prop:dim4} 
that its dimension is smaller or equal to 4. Only the situation $x\in\{j-k-\frac{1}{2},j+k+\frac{1}{2}\}$ remains to be treated.

From the Bratteli diagram, we see that the characteristic equation for $C-A$ in \eqref{eq:quo5} implies that:
\[\prod_{l=j-k-\frac{1}{2}}^{j+k-\frac{1}{2}}\bigl(C-A-l(l+1)+(j-\frac{1}{2})(j+\frac{1}{2})\bigr)
\prod_{l=j-k+\frac{1}{2}}^{j+k+\frac{1}{2}}\bigl(C-A-l(l+1)+(j+\frac{1}{2})(j+\frac{3}{2})\bigr)=0\,.\]

Now, let $x=j+k+\frac{1}{2}$, so that we replace $C$ by $(j+k+\frac{1}{2})(j+k+\frac{3}{2})$. It is easy to see from the equation just 
above that the eigenvalue $(j-\frac{1}{2})(j+\frac{1}{2})$ for $A$ is excluded.

Similarly, let $x=j-k-\frac{1}{2}$, so that we replace $C$ by $(j-k-\frac{1}{2})(j-k+\frac{1}{2})$. It is easy to see that the eigenvalue $(j+\frac{1}{2})(j+\frac{3}{2})$ for $A$ is excluded.

In both cases we conclude that in $\overline{\mathcal{R}}\bigl(j(j+1),\frac{3}{4},k(k+1)\bigr)_x$, the generator $A$ is a number. Then if $A$ is equal to the number $a$, collecting the terms 
with $B$ in relation \eqref{eq:UR1mod}, we have that 
\[(a^2-2(x^2-1)a+x^2(x^2-1))B\]
must also be a number. Using that $a$ is a root of the characteristic polynomial of $A$, this gives that $2aB$ is a number. As $a$ is not zero, this gives that $B$ is also equal to a number.

Finally, in both cases, we have that the relations force $A$ and $B$ to be equal to numbers, so that the dimension must be less or equal to 1.
\end{proof}

\begin{rem}
The above proof shows that both relations \eqref{eq:quo2}--\eqref{eq:quo3} can be removed from the definition of the quotient $\overline{\mathcal{R}}\bigl(j(j+1),\frac{3}{4},k(k+1)\bigr)$. 
Indeed they were not used at all. More obviously, one of the two relations in \eqref{eq:quo5} can also be removed. Here we remind the reader that we exclude the situation $(j,k)=(\frac{1}{2},\frac{1}{2})$ 
which was treated earlier. If $j=k=\frac{1}{2}$ this remark does not hold anymore.
\end{rem}

\begin{rem}[The situation $j=k$]$\ $

$\bullet$ In this remark, we assume that $j=k$. We still denote $x=j+\frac{1}{2}$ and we set also
\[z:=\sqrt{x^2-(c+\frac{1}{4})}\ .\]
We make the following change of generators in $H_{j,j,c}$:
\[\mathcal{A}=A+z-x^2\ \ \ \ \ \ \ \text{and}\ \ \ \ \ \ \ \ \mathcal{B}=B+z-x^2\ .\]
Then one can check with straightforward calculations that the defining relations of $H_{j,j,c}$ in terms of $\cA$ and $\cB$ can be written:
\begin{eqnarray}
  && \cA^2=2z\cA+x^2-z^2\quad ,\qquad  \cB^2=2z\cB+x^2-z^2\ ,\\
  &&\cA\cB\cA=\cB\cA\cB\ ,\\
 && \cA\cB\cA=(z-1)(\cA\cB+\cB\cA)+(-z^2+2z-x^2)(\cA+\cB)+z^3-3z^2+3x^2z-x^2\ .
\end{eqnarray}
In particular, $\cA$ and $\cB$ satisfy the braid relations so that the algebra $H_{j,j,c}$ becomes a quotient of algebra of the braid group (on 3 strands). 
Since $\cA$ and $\cB$ satisfy a quadratic relation, this quotient factors through the Hecke algebra. It is interesting to note that the situation $c=(2j+\frac{1}{2})(2j+\frac{3}{2})$ 
(which corresponds to the more delicate situation of a multiplicity equal to 1 in the proof above) corresponds in fact to a non semisimple regime for the Hecke algebra.

$\bullet$ Let $H_{j}$ be the quotient of the universal Racah algebra $\mathcal{R}(j(j+1),\frac{3}{4},j(j+1))$ by the quadratic relations \eqref{quadr1} for $A$ and $B$ (here $y=x$). In other words, we put back the central element $C$ in $H_{j,j,c}$. Then the change of variables in the preceding item shows that we can find elements $\cA$ and $\cB$ satisfying the braid relations. It would be interesting to study the quotient of the algebra of the braid group that one obtains this way (note that since $C$ appears in the change of variables, the new elements $\cA$ and $\cB$ no longer satisfy a quadratic characteristic equation).
\end{rem}

\section{Conclusion and perspectives}

In this paper, we have proposed a conjecture regarding the relation between  quotients of the Racah algebra and the centralizers
$\text{End}_{\mathfrak{su}(2)}( [2j_1] \otimes [2j_2] \otimes [2j_3] )$ for any choice of three finite irreducible representations of $\mathfrak{su}(2)$.
It provides a description in terms of generators and relations of each of these centralizers. The conjecture has been proven in different cases: previously known results 
have been recovered and descriptions have been found for new cases. It would obviously be desirable to provide a proof of this conjecture in general.
We believe that it should always be possible to simplify
the presentation of the quotiented Racah algebra described in Conjecture \ref{conj:t} by removing the central element $C$.
There may also exists a diagrammatic presentation of the defining relations of these quotients as in the case of
the Temperley--Lieb or Brauer algebras.

Three directions for generalizing the conjecture that can be envisaged. The first consists in increasing the number of tensor product and in 
considering the $N$ fold tensor product of 
$\mathfrak{su}(2)$. In this case the Racah algebra is replaced by the higher rank algebra introduced in \cite{DV}. 
The quotient which gives the centralizer may be associated 
to a Bratteli diagram with $N$ rows describing the direct sum decomposition of $N$ fold tensor product.
For the particular cases of $N$ fundamental representations, we must find that the quotient is isomorphic to the Temperley--Lieb algebra $TL_N(1)$
or for $N$ spin-1 representations to the Brauer algebra $B_N(3)$. The results for the three fold tensor product should be the building blocks to
obtain the presentation of the centralizer for the $N$ fold tensor product. In particular, our new results obtained in Sections \ref{sec:sff} and \ref{sec:f11}
should be useful to describe the centralizer for the $N$ fold tensor product with some spin-$\frac{1}{2}$ representations and some spin-$1$ representations.

The second generalization is to consider algebras other than $\mathfrak{su}(2)$. In the case of the super-algebra $\mathfrak{osp}(1|2)$, the
Bannai--Ito algebra plays the role of the Racah algebra and for three fundamental representation a quotient is isomorphic to the Brauer algebra \cite{CFV}.
Let us remark that 
the generalizations of the Racah algebra are not known say
for the Lie algebras $\mathfrak{su}(n)$  with $n\geq 3$. 
Their study would be a pre-requisite to the characterization following the lines of the present paper of the centralizers of tensor product of irreducible representations.
The construction of 
the Racah algebra 
associated to $\mathfrak{su}(n)$ should also prove important in the study of orthogonal polynomials.

The third generalization is to examine the $q$-deformation of our results. Indeed, we could consider the quantum group $U_q(\mathfrak{su}(2))$
instead of the Lie algebra $\mathfrak{su}(2)$. The Racah algebra would be replaced by the Askey--Wilson algebra \cite{Z91}.
For three fundamental representations, a quotient of the Askey--Wilson algebra must be isomorphic to the Temperley--Lieb algebra for $q\neq 1$
and for three spin-$1$, it must be isomorphic to the Birman--Murakami--Wenzl algebra \cite{BMW} according to the results in \cite{LZ,LZ2}.
 We trust that our result of Section \ref{sec:sff} can also be $q$-deformed: in fact
the one-boundary Temperley--Lieb algebra \cite{MS, MW,NRG} (see equations \eqref{eq:1btl}) is the centralizer 
for one spin-$j$ representation and two spin-$\frac{1}{2}$ representations of $U_q(\mathfrak{su}(2))$.

Finally, it is possible to think of generalizations mixing some of the extension avenues presented above. For instance,
the higher rank Askey-Wilson algebra introduced in \cite{PW} could be the starting point to deal with the $N$ fold tensor product of 
representations of $U_q(\mathfrak{su}(2))$ and the higher rank Bannai--Ito \cite{DGV} with the $N$ fold tensor product of 
representations of $\mathfrak{osp}(1|2)$.\\[1cm]

\textbf{Acknowledgements.} N.Cramp\'e and L.Poulain d'Andecy are partially supported by Agence National de la Recherche Projet AHA ANR-18-CE40-0001.
 N.Cramp\'e is gratefully holding a CRM--Simons professorship. L.Poulain d'Andecy warmly thanks the Centre de Recherches Math\'ematiques (CRM) for 
 hospitality and support during his visit to Montreal in the course of this investigation
 The research of L.Vinet is supported in part by a Discovery Grant from the Natural Science and Engineering Research Council (NSERC) of Canada.


\begin{thebibliography}{99}

\bibitem{Brauer} R. Brauer, 
\textsl{On algebras which are connected with the semisimple continuous groups,} 
Ann. of Math. 38:857--872, 1937.

\bibitem{BMW} J.Birman, and H.Wenzl, 
\textsl{Braids, link polynomials and a new algebra,} 
Trans. Am. Math. Soc. 313:249--273, 1989.

\bibitem{CFV} N. Crampe, L. Frappat, and L. Vinet,
\textsl{Bannai--Ito and Brauer algebras,} In preparation.

\bibitem{DV} H. De Bie, V.X. Genest, W. van de Vijver, and L. Vinet,
{\sl A higher rank racah algebra and the $\mathbb{Z}_n^2$ Laplace-Dunkl operator,}
J. Phys. A 51:025203, 2018  and \texttt{arXiv:1610.02638}.

\bibitem{DGV}H. De Bie, V. X. Genest, L. Vinet,
\textsl{The $\mathbb{Z}_n^2$ Dirac-Dunkl operator and a higher rank Bannai--Ito algebra,}
Adv. Math. 303:390--414, 2016 and \texttt{arXiv:1511.02177}.

\bibitem{GZ} Ya. A. Granovskii, and A.S. Zhedanov,
\textsl{Nature of the symmetry group of the $6j$-symbol,}
JETP 67:1982--1985, 1988.

\bibitem{Jimbo}M. Jimbo,
\textsl{A q-Analogue of $U(gl(N + 1))$, Hecke Algebra, and the Yang--Baxter Equation, }
Lett. Math. Phys. 11:247--252, 1986.

\bibitem{LZ} G.I. Lehrer, and R.B. Zhang, 
\textsl{On endomorphisms of quantum tensor space,}
Lett. Math. Phys. 86:209--227, 2008  and \texttt{arXiv:0806.3807}.

 \bibitem{LZ2} G.I. Lehrer, and R.B. Zhang,
\textsl{A Temperley-Lieb analogue for the BMW algebra,}
\texttt{arXiv:0806.0687}.

\bibitem{MS} P.P Martin, and H. Saleur, 
\textsl{On an algebraic approach to higher dimensional statistical mechanics,}
Commun. Math. Phys. 158:155--190, 1993 and \texttt{ hep-th/9208061}.

\bibitem{MW} P.P. Martin, and D. Woodcock, 
\textsl{On the structure of the blob algebra,}
Journal of Algebra 225:957--988 (2000).

\bibitem{NRG} A. Nichols, V. Rittenberg, and J. de Gier, 
\textsl{One-boundary Temperley-Lieb algebras in the XXZ and loop models,}
J. Stat. Mech. 0503:P03003, 2005 and \texttt{arXiv:cond-mat/0411512}.

\bibitem{PW}S. Post, and A. Walter,
 {\sl  A higher rank extension of the Askey-Wilson Algebra,}
 \texttt{arXiv:1705.01860}.
 
\bibitem{Rac} G. Racah,
\textsl{Theory of complex spectra II,}
Phys. Rev. 62:438--462, 1942.

\bibitem{TL} N. Temperley, and E. Lieb, 
\textsl{Relations between the 'Percolation' and 'Colouring' Problem and other Graph-Theoretical Problems Associated with Regular Planar Lattices: Some Exact Results for the 'Percolation' Problem.}
Proc. Royal Soc. A 322:251--280, 1971.


\bibitem{Z91} A.S. Zhedanov, 
{\sl “Hidden symmetry” of the Askey-Wilson polynomials,}
Theor. Math. Phys. 89:1146--1157, 1991.





\end{thebibliography}
\end{document}